\documentclass[12pt]{amsart}
\usepackage{amsmath,amssymb,amsthm,amscd}
\usepackage[matrix,arrow,curve]{xy}

\usepackage{graphicx}

\newtheorem{formula}{}[section]
\newtheorem{proposition}[formula]{Proposition}
\newtheorem{corollary}[formula]{Corollary}
\newtheorem{lemma}[formula]{Lemma}
\newtheorem{theorem}[formula]{Theorem}

\newtheorem*{theorem*}{Theorem}

\theoremstyle{definition}
\newtheorem{definition}[formula]{Definition}
\newtheorem{construction}[formula]{Construction}
\newtheorem{example}[formula]{Example}
\newtheorem{notation}[formula]{Notation}

\theoremstyle{remark}
\newtheorem{remark}[formula]{Remark}

\newtheorem{question}[formula]{Question}
\newtheorem*{problem*}{Problem}

\begin{document}

\title[Geometrization of $3$-manifolds defined by vector-colourings]{Canonical geometrization of orientable 
$3$-manifolds defined by vector-colourings of $3$-polytopes}
\author[N.Yu.~Erokhovets]{Nikolai~Erokhovets}
\address{Steklov Mathematical Institute of Russian Academy of Sciences, Moscow, Russia}
\email{erochovetsn@hotmail.com}

\def\sgn{\mathrm{sgn}\,}
\def\bideg{\mathrm{bideg}\,}
\def\tdeg{\mathrm{tdeg}\,}
\def\sdeg{\mathrm{sdeg}\,}
\def\grad{\mathrm{grad}\,}
\def\ch{\mathrm{ch}\,}
\def\sh{\mathrm{sh}\,}
\def\th{\mathrm{th}\,}
\def\RZ{\mathbb{R}\mathcal{Z}}
\def\mod{\mathrm{mod}\,}
\def\In{\mathrm{In}\,}
\def\Im{\mathrm{Im}\,}
\def\Ker{\mathrm{Ker}\,}
\def\Hom{\mathrm{Hom}\,}
\def\Tor{\mathrm{Tor}\,}
\def\rk{\mathrm{rk}\,}
\def\codim{\mathrm{codim}\,}

\def\ko{{\mathbf k}}
\def\sk{\mathrm{sk}\,}
\def\RC{\mathrm{RC}\,}
\def\gr{\mathrm{gr}\,}

\def\R{{\mathbb R}}
\def\C{{\mathbb C}}
\def\Z{{\mathbb Z}}
\def\A{{\mathcal A}}
\def\B{{\mathcal B}}
\def\K{{\mathcal K}}
\def\M{{\mathcal M}}
\def\N{{\mathcal N}}
\def\E{{\mathcal E}}
\def\G{{\mathcal G}}
\def\D{{\mathcal D}}
\def\F{{\mathcal F}}
\def\L{{\mathcal L}}
\def\V{{\mathcal V}}
\def\H{{\mathcal H}}




\subjclass[2010]{
52B70, 
53C15,  
52B10, 
57R91
}


\begin{abstract}
In short geometrization conjecture of W.\,Thurston (finally proved by G.~Perelman) says that
any oriented $3$-manifold can be canonically partitioned into pieces, which have a geometric structure of one of the eight types.
In the seminal paper (1991) M.\,W.\,Davis and T.\,Januszkiewicz introduced a wide class of $n$-dimensional manifolds
-- small covers over simple $n$-polytopes. 
We give a complete answer to the following problem: to build an explicit canonical decomposition for any orientable $3$-manifold defined by a vector-colouring of a simple $3$-polytope, in particular for a small cover. 
The proof is based on analysis of results in this direction obtained before  by different authors.

\end{abstract}
\maketitle

\setcounter{section}{0}

\section*{Introduction}
Toric topology (see \cite{BP15}) allows one to give explicit examples for deep results in various fields of mathematics.
In this paper we consider Thurston's geometrization conjecture, which was finally proved by G.\,Perelman.
To formulate the precise result we will need some definitions. We follow the exposition in the book \cite{AFW15}. All the 
precise references and additional details can be found there. Also we recommend the surveys \cite{B02}, 
\cite{S83}, and \cite{T02}.

\begin{definition}
A {\it $3$-dimensional geometry} is a smooth, simply connected $3$-manifold $X$ equipped
with a smooth, transitive action of a Lie group $G$ by diffeomorphisms on $X$, with
compact point stabilizers.  The group $G$ is required to be maximal among all such groups for~$X$.

A {\it geometric structure} on a $3$-manifold $N$ (modeled on $X$) is a diffeomorphism from
the interior of $N$ to $X/\pi$, where $\pi$ is a discrete subgroup of $G$ acting freely on $X$.

Thurston showed that, up to a certain equivalence, there exist precisely eight
$3$-dimensional geometries that model compact $3$-manifolds. These are:
the sphere $S^3$, the Euclidean space $\mathbb R^3$, the hyperbolic space 
$\mathbb{H}^3$, $S^2\times \mathbb R$, $\mathbb H^2\times\mathbb  R$, 
the universal cover $\widetilde{SL(2,\mathbb{R})}$ of $SL(2,\mathbb {R})$, and two further geometries called
$Nil$ and $Sol$. We will say that $N$ has a geometric structure of finite volume, if $X/\pi$ has a finite volume.
\end{definition}

After Perelman's work geometrization theorem conjectured by W.~Thurston can be formulated in the following way 
(see \cite[Theorem 1.9.1]{AFW15} and \cite[Introduction]{MT14}). 
\begin{theorem}[Geometric decomposition theorem]\label{GDT}
Let $N$ be a compact closed, orientable, irreducible $3$-manifold. 
There is a (possibly empty) collection of disjointly
embedded incompressible surfaces $S_1$, $\dots$, $S_m$, which are either tori or Klein
bottles, such that each component of $N$ cut along $S_1\cup\dots\cup S_m$ has a geometric structure. Any such
collection with a minimal number of components is unique up to isotopy.
\end{theorem}

In \cite{DJ91} M.\,W.\,Davis and T.\,Januszkiewicz introduced a class of $n$-manifolds called
{\it small covers} over simple $n$-polytopes. Each manifold of this type is glued of $2^n$ 
copies of the polytope $P$ and is defined by a mapping $\Lambda$ from the set of facets $\{F_1,\dots, F_m\}$ 
of $P$ to $\mathbb Z_2^n=(\mathbb Z/2\mathbb Z)^n$ such that
for any face the images of  the facets containing it are linearly independent.  
In \cite[Example 1.21]{DJ91} there is a sketch of the proof that geometric decomposition {\bf exists}
for $3$-manifolds realizable as small covers over simple $3$-polytopes: 
{\it $\langle \dots\rangle$ let us consider the $3$-manifolds which arise as small covers of such a $P^3$. 
If $P^3$ has no triangles or squares as faces, then it follows from Andreev's Theorem 
(cf. [Andreev] or [Thurston]) that it can be realized as a right-angled polytope in the hyperbolic
$3$-space. Hence, any small cover carries a hyperbolic structure. If $P^3$ has no triangular
faces, $M^3$ can be decomposed into Seifert fibered or hyperbolic pieces glued
along tori or Klein bottles arising from square faces. In particular, such an $M^3$ is
aspherical. If $P^3$ has triangular faces, then $M^3$ has a decomposition into such pieces
glued along projective planes. Of course, all this fits with Thurston's Conjecture.}

Actually, any mapping $\Lambda\colon \{F_1,\dots, F_m\}\to \mathbb Z_2^r$ such that for any face the images of the facets
containing it are linearly independent, gives a manifold $M(P,\Lambda)$, which we call
a {\it manifold defined by a vector-coloring} $\Lambda$. If $r=n$, then $M(P,\Lambda)$ is a small cover, and if $r=m$ 
and the images of all the facets are linearly independent, 
then $M(P,\Lambda)$ is a real moment-angle manifold $\mathbb R\mathcal{Z}_P$.

In this paper we solve the following
\begin{problem*}
To find {\bf explicitly} the minimal geometric decomposition of any oriented $3$-manifold $M(P,\Lambda)$ defined by 
a vector-colouring $\Lambda$ of a simple $3$-polytope. In particular, of any small cover and real moment-angle manifold
over a simple $3$-polytope.
\end{problem*}

The answer in given in Theorem \ref{M-th}.

It turned out, that to give the explicit minimal geometric decomposition, 
additional arguments should be used. Namely, one needs to
\begin{enumerate}
\item  use {\it $3$-belts} and {\it $4$-belts} in addition to triangular and quadrangular facets (for simplicial polytopes this
corresponds to $3$-cycles not bounding triangles and chordless $4$-cycles);
\item use the canonical decomposition of a simple $3$-polytope into a connected sum along vertices of simplices and {\it flag} $3$-polytopes 
(that is, simple $3$-polytopes different from the simplex and having no $3$-belts).
This corresponds to the Kneser-Milnor decomposition of oriented $3$-manifolds into connected sums of 
prime manifolds;
\item  use right-angled hyperbolic polytopes of finite volume and simple polytopes 
obtained by cutting off their ideal vertices (such polytopes were studied in \cite[Section 10.3]{DO01} and \cite{E19}, we call them together with the cube and the $5$-prism  {\it almost Pogorelov polytopes}); 
\item  find a canonical "minimal" decomposition of a flag simple $3$-polytope along $4$-belts. This decomposition is a 
special case of a more general decomposition of a so-called Coxeter orbifold in \cite{S09}.
It turned out, that there may be many non-equivalent ways to cut a flag polytope along $4$-belts into almost Pogorelov polytopes.
Such decompositions are used in \cite{DO01} (see also \cite[after Theorem 20.5.2]{D08})  to prove the Singer conjecture on 
the vanishing of the reduced ${\it l^2}$-homology except in the middle
dimension for closed aspherical manifolds for the special case of $3$-manifolds arising from right-angled Coxeter groups. 
The paper \cite{S09} also is motivated by this topic.
The canonical decomposition of a flag $3$-polytope corresponds to the $JSJ$-decomposition of the oriented 
irreducible $3$-manifold $M(P,\Lambda)$. The $JSJ$-tori, which  
are defined up to an isotopy, correspond to the "canonical" $4$-belts and to the quadrangles of almost Pogorelov polytopes which 
are not used in the decomposition.  Each "free quadrangle" of an almost Pogorelov polytope corresponds to 
incompressible submanifolds, which is either tori $T^2$ (if the corresponding vector $\Lambda_i$ does not lie in the subspace
spanned by the vectors of adjacent facets), or the Klein bottles $K^2$ (if $\Lambda_i$ lies in the subspace
spanned by the vectors of adjacent facets), and the $4$-belt around the quadrangle 
corresponds to the boundaries of
the tubular neighbourhoods of this submanifolds.  
The tori $T^2$ are $JSJ$-tori themselves, and in the case of $K^2$
the boundaries of the tubular neighbourhoods are $JSJ$-tori; 
\item  prove that the arising tori are incompressible. For this we use  a retraction of a real moment-angle complex
to a subcomplex corresponding to the full subcomplex. This method 
was communicated to the author by Taras Panov (see \cite[Exercise 4.2.13]{BP15} 
and \cite[Proposition 2.2]{PV16}).
An alternative way of reasoning using an explicit description of the fundamental group can be found in \cite{W19,WY17,LW20}.
For submanifolds corresponding to facets the incompressibility is proved as \cite[Theorem 3.3]{WY17}. Another approach to this result is presented in \cite{DO01,DJS98,D08} and is based on the fact that  a facet  submanifold
is a totally geodesic hypersurface in $M(P,\Lambda)$ 
with the structure of a cubical complex, which is nonpositively curved in the sense of Alexandrov and Gromov \cite{G87}.

As it was mentioned above some of the $JSJ$-tori corresponding to $4$-belts around 
quadrangles bound orientable manifolds, which are $I$-bundles over $K^2$. 
These pieces can be endowed with a Euclidean geometry with infinite volume. To reduce the number of geometric pieces and 
to make all of them of finite volume these tori should be replaced with the corresponding Klein bottles; 
\item write explicitly a geometrization of the arising pieces. This can be done using the 
construction invented by A.D.~Mednykh and A.Yu.~Vesnin \cite{MV86, V87, M90, VM99, V17} in the case of right-angled
polytopes of finite volume in $\mathbb R^3$, $S^3$, $\mathbb{H}^3$, $S^2\times \mathbb R$ and $\mathbb{H}^2\times \mathbb R$. The construction builds a subgroup of the right-angled Coxeter group generated by reflections in facets of the polytope.
This subgroup is defined by the characteristic function and acts freely on the space. The orbit space can be identified with a piece of the manifold. In our case each flag simple $3$-polytope is canonically decomposed along $4$-belts into polytopes which can 
be realized as right-angled polytopes in
these five geometries (for pieces corresponding to hyperbolic geometry we also delete all the remaining quadrangles). 
For geometries $\mathbb{H}^3$ and $S^3$ such a realization is unique up to isometries, and for 
$\mathbb{H}^2\times \mathbb R$, $\mathbb R^3$, and $S^2\times \mathbb R$ such a realization is not unique.
\end{enumerate}

Prime orientable $3$-manifolds $M(P,\Lambda)$ correspond to 
\begin{enumerate}
\item the simplex $\Delta^3$: $M(P,\Lambda)$ is either $S^3$, or $\mathbb RP^3$. Both of them are irreducible; 
\item the $3$-prism $\Delta^2\times I$:  $M(P,\Lambda)$ is either $S^2\times S^1$, which is prime, or  
$\mathbb{R}P^3\#\mathbb{R}P^3$, which is not prime;
\item flag simple $3$-polytopes: $M(P,\Lambda)$ is aspherical, therefore irreducible. 
\end{enumerate}

On orientable manifolds $M(P,\Lambda)$ we obtain geometric structures of finite volume modelled~on:

\begin{itemize}
\item $S^3$:\quad $\Delta^3$ can be realized as a right-angled polytope in $S^3$
and a realization is unique (up to isometries of $S^3$).  Each $M(P,\Lambda)$ is a
compact closed spherical manifold $S^3$ or $\mathbb RP^3$.

\item $S^2\times \mathbb R$: \quad $\Delta^2\times I$ can be realized as a right-angled polytope in $S^2\times \mathbb R$. A right-angled realization
of $\Delta^2\subset S^2$ is unique (up to isometries), while for $I$ in $\mathbb R$ it is not. Each $M(P,\Lambda)$ is a
compact closed manifold $S^2\times S^1$ or $\mathbb RP^3\#\mathbb RP^3$ with a geometric structure modelled on $S^2\times \mathbb R$.

\item $\mathbb R^3$,  $\mathbb H^2\times \mathbb R$, $\mathbb{H}^3$: \quad A flag simple $3$-polytope is either the cube $I^3$ 
(the $4$-prism), or the $k$-prism with $k\geqslant 5$, or a Pogorelov polytope (that is, a flag polytope without $4$-belts), 
or none of the above. 

\begin{itemize}
\item $\mathbb R^3$:\quad the cube $I^3$ can be realized as a right-angled polytope in $\mathbb R^3$, 
and a realization is not unique. We have $\mathbb R\mathcal{Z}_{I^3}=T^3=S^1\times S^1\times S^1$, and any manifold $M(I^3, \Lambda)$ is a compact closed  Seifert fibered manifold with a Euclidean geometry. 

\item $\mathbb H^2\times \mathbb R$:\quad the $k$-prism, $k\geqslant 5$, can be realized as a right-angled polytope in $\mathbb H^2\times\mathbb R$, and 
a realization is not unique. Each $M(P,\Lambda)$ is a compact closed Seifert fibered manifold with a geometric structure modelled on $\mathbb H^2\times \mathbb R$.

\item $\mathbb{H}^3$:\quad a Pogorelov polytope can be realized as a bounded right-angled polytope in $\mathbb H^3$, and a realization is unique (up to isometries).
Each $M(P,\Lambda)$ is a compact closed hyperbolic manifold.

\item $\mathbb H^2\times \mathbb R$, $\mathbb{H}^3$:\quad if a flag $3$-polytope is not a $k$-prism and not a Pogorelov polytope, then, as it was mentioned above, it can be canonically 
cut along $4$-belts into pieces which are either $k$-prisms, $k\geqslant 5$, without a set of disjoint quadrangles or almost Pogorelov polytopes without a set of disjoint quadrangles corresponding to the "canonical" belts. 

\begin{itemize}
\item A $k$-prism, $k\geqslant 5$, can be realized as an (unbounded) right-angled polytope of finite volume in 
$\mathbb H^2\times\mathbb R$ and each deleted quadrangle corresponds 
to a product of an ideal vertex and the segment. A realization is not unique. The corresponding pieces of $M(P,\Lambda)$
are Seifert fibered manifolds with geometric structures modelled on $\mathbb H^2\times\mathbb R$.

\item An almost Pogorelov polytope with all its quadrangles (both corresponding to "canonical" $4$-belts and "free") deleted 
can be realized as an (unbounded) right-angled polytope of finite volume in $\mathbb{H}^3$, and each deleted quadrangle corresponds to an ideal vertex. A realization is unique (up to isometries). The corresponding pieces of $M(P,\Lambda)$ are 
hyperbolic manifolds of finite volume.
\end{itemize}
\end{itemize}
\end{itemize}

The paper is organized as follows.

In Section \ref{Sec:RMAM} we give definitions and  basic facts on manifolds defined by vector-colourings.
In particular, in Proposition \ref{orprop} we give a criterion when the manifold $M(P,\Lambda)$ is orientable (generalizing the case
of small covers considered in \cite{NN05}), and in Proposition \ref{doubleprop} 
for a non-orientable manifold $M(P,\Lambda)$ we construct the orientable double cover $M(P,\widehat{\Lambda})$. 
We consider the partially ordered set $\mathcal{F}(P)$ of subgroups 
$H(\Lambda)\subset \mathbb Z_2^m$ acting freely on $\mathbb R\mathcal{Z}_P$ and in Example \ref{exas}
present a subgroup of dimension less than $m-3$ corresponding to a maximal element.

In Section \ref{Sec:bo} we develop the technique of $k$-belts of simple $3$-polytopes. We introduce the notions of 
a nested family of belts and nested family of curves corresponding to it. In Lemma \ref{NFC-lemma} we prove
that any nested family of belts admits a nested family of curves, and in Lemma \ref{lem:iso} -- that any two nested
families of curves corresponding to the same family of belts are isotopic in this class. 
We consider the operations of cutting a $3$-polytope along a belt, 
a connected sum of two $3$-polytopes along vertices and along facets surrounded by belts. In Corollaries \ref{3bplane}
and \ref{4bplane} we prove that for any nested family of $3$-belts and any nested family of $4$-belts there is a geometric
realization of a $3$-polytope $P$ such that disjoint planar sections of $P$ give the corresponding nested families of curves.  

In Section \ref{Sec:prime} we build the prime decomposition of any orientable $3$-manifold $M(P,\Lambda)$. First, in Proposition 
\ref{CSprop} we express the manifold corresponding to a connected sum of simple $3$-polytopes along  
vertices as a connected sum of copies of the corresponding manifolds and copies of $S^2\times S^1$. Then in Theorem \ref{3th}
we prove the main result of this section.

In Section \ref{Sec:main} we prove the main results of the article. First, in Subsections \ref{JSJgen} and \ref{JSJgeom} we present
a general information on a $JSJ$-decomposition and a geometrization. Then in Subsection \ref{MPJSJ} on the base of this information we give the $JSJ$-decomposition and the minimal geometric decomposition for any irreducible orientable $3$-manifold 
$M(P,\Lambda)$. The main result of the article is Theorem \ref{M-th}. It is proved in several steps. First, in Subsubsection
\ref{flagdec} we build a canonical decomposition of a flag $3$-polytope along $4$-belts and prove its uniqueness. Then, 
in Subsubsection \ref{BFinc} we study submanifolds corresponding to belts and facets surrounded by belts.  
In Propositions \ref{Binprop} and \ref{Finprop} we describe them explicitly and prove that they are incompressible. In Subsubsection \ref{JSJlast} we finish the proof of Theorem \ref{M-th}. In Proposition \ref{MPprop} we 
describe explicitly the pieces corresponding to $k$-prisms, in particular, their Seifert fibered structure. Then we prove that the constructed family of tori is indeed the family of all the $JSJ$-tori. At last, we prove that $M(P,\Lambda)$ can not be a $Sol$-manifold and build an explicit geometrization using Construction \ref{MVconstr} by A.D.\,Mednykh and Yu.\,Vesnin.

\section{Manifolds defined by vector-colourings}\label{Sec:RMAM}
For an introduction to the polytope theory we recommend \cite{Z95}.  
In this paper by a {\it polytope} we call an $n$-dimensional combinatorial convex polytope (in most cases $n$ will be equal to $3$). Sometimes we implicitly use  its geometric realization in $\mathbb R^n$ and sometimes we use it explicitly. In the latter case we call the polytope {\it geometric}.
A polytope is simple, if any its vertex is contained in exactly $n$ facets.
Let $\{F_1,\dots,F_m\}$ be the set of all the facets, and 
$\mathbb Z_2=\Z/2\Z$.
\begin{definition} 
For each geometric simple polytope $P$ one can associate an $n$-dimensional
{\it real moment-angle manifold}:
$$
\RZ_P=P\times \mathbb Z_2^m/\sim, \text{ where }(p,a)\sim(q,b)\text{ if and only if }p=q \text{ and }a-b\in\langle e_i\colon p\in F_i\rangle,
$$ 
where $e_1,\dots, e_m$ is a basis in $\mathbb Z_2^m$ as a vector space.
\end{definition}
The space $\mathcal{Z}_P$ was introduced in \cite{DJ91}. 
It is convenient to imagine  $\RZ_P$ as a space glued from copies of the polytope $P$
 along facets. If we fix an orientation on $P\times 0$, then define on the polytope $P\times a$ the same orientation, 
 if $a$ has an even number of unit coordinates, and the opposite orientation, in the other case.  
 A polytope $P\times a$ is glued to the polytope $P\times (a+e_i)$ along the facet $F_i$. At each vertex 
 the polytopes are arranged as coordinate orthants in $\mathbb R^n$, at each edge -- as the orthants at a coordinate axis,
 and at face of dimension $i$ -- as the orthants at an $i$-dimensional coordinate subspace. Therefore,
 $\RZ_P$ has a natural structure of an oriented piecewise-linear manifold.  There is a natural action of $\mathbb Z_2^m$
 on $\RZ_P$ induced from the action on the second factor. The actions of basis vectors $e_i$ can be viewed as
 reflections in facets of the polytope.
\begin{example}\label{ZPkex}
For a $k$-gon $P_k$  the closed orientable $2$-dimensional manifold $\mathbb R\mathcal{Z}_{P_k}$ is homeomorphic to 
a sphere with $g$ handles.  The genus $g$ can be calculated using the Euler characteristic: 
$$
2-2g=\chi(\mathbb R\mathcal{Z}_{P_k})=f_0-f_1+f_2=k2^{k-2}-k2^{k-1}+2^k.
$$
Thus, $g=(k-4)2^{k-3}+1$.
\end{example}

There is another representation of $\RZ_P$. By a {\it simplicial complex $K$} 
on the set of vertices $[m]=\{1,\dots,m\}$ we mean an abstract simplicial complex, 
that is a collection of subsets $\sigma\subset[m]$ such that
for any $\sigma\in K$ and any $\tau\subset\sigma$ we have $\tau\in K$. The subsets $\sigma\in K$ are called {\it simplices}.

Let $D^1=[-1,1]$ and $S^0=\{-1,1\}=\partial D^1$. 
\begin{definition}
A {\it real moment-angle complex} 
of a simplicial complex $K$  is defined as:
$$
\mathbb{R}\mathcal{Z}_K=\bigcup\limits_{\sigma\in K}(D^1,S^0)^\sigma, \text{ where }(D^1,S^0)^\sigma=X_1\times\dots\times  X_m,\;X_i=\begin{cases} D^1,&i\in \sigma;\\
S^0,&i\notin\sigma\end{cases}.
$$
\end{definition}
There is a natural action of $\mathbb Z_2^m$ on $\mathbb R\mathcal{Z}_K$, which arises from the identification of the 
multiplicative group $\{-1,1\}$ with the additive group $\Z_2=\{0,1\}$. This group acts by changing the signs of the coordinates. 

Each simple polytope $P$ corresponds to a simplicial complex $K_P=\{\sigma\subset [m]\colon \bigcap_{i\in \sigma}F_i\ne\varnothing\}$, which is isomorphic to the boundary complex $\partial P^*$ of the dual simplicial polytope.  

\begin{proposition}[see \cite{BP15}, and also \cite{BE17S}]\label{RZKP}
There is an equivariant homeomorphism 
$$
\RZ_{K_P}\simeq \mathbb R\mathcal{Z}_P.
$$ 
\end{proposition}
This homeomorphism  implies that
the topological type of $\RZ_P$ does not depend on the geometric realization of $P$.
\begin{proof}[Sketch of the proof of Proposition \ref{RZKP}]
Consider the {\it barycentric embedding} $b\colon P\to [0,1]^m$, which is a piecewise 
linear mapping defined on vertices of the barycentric subdivision of $P$: 
$$
b(v_{G})=(y_1,\dots,y_m), \text{ where $v_G$ is the barycentre of a face $G$, and }y_i=\begin{cases} 0,&G\subset F_i;\\
1,&G\not\subset F_i.\end{cases}
$$
Denote by $[0,1]^P$ the image of this embedding. Then a homeomorphism 
$\mathbb R\mathcal{Z}_P\to \RZ_{K_P}$ is given by the mapping $(x,a)\to a(b(x))$,
where $a(b(x))$ is the result of the action of $a\in\mathbb Z_2^m$ on $b(x)\in [0,1]^P\subset \mathbb{R}\mathcal{Z}_{K_P}$.
\end{proof}

We consider manifolds obtained as orbit spaces of free actions of subgroups $H\subset\mathbb Z_2^m$ on 
$\mathbb R\mathcal{Z}_P$. Each subgroup of $\mathbb Z_2^m$ is isomorphic to $\mathbb Z_2^{m-r}$ for some $r$ and may 
be defined as a kernel of a an epimorphism $\Lambda\colon\mathbb Z_2^m\to\mathbb Z_2^r$. 
Such a mapping is uniquely defined by the images $\Lambda_i\in\mathbb Z_2^r$ of all the vectors 
$e_i\in\mathbb Z_2^m$ corresponding to facets $F_i$, $i=1$,\dots, $m$. 
It can be shown that the action of the subgroup $H(\Lambda)\subset \mathbb Z_2^m$ on $\mathbb R\mathcal{Z}_P$ is free
if and only if 
$$
(*)\quad\text{for any vertex $F_{i_1}\cap\dots\cap F_{i_n}$ of $P$ the vectors $\Lambda_{i_1}$, $\dots$,
$\Lambda_{i_n}$ are linearly independent.} 
$$
\begin{definition}
We call a mapping $\Lambda\colon \{F_1,\dots,F_m\}\to \mathbb Z_2^r$  such that the images $\Lambda_j$
of the facets $F_j$ span $\mathbb Z_2^r$ and $\Lambda$ satisfies (*) {\it a vector-colouring} of rank $r$. 
\end{definition}
\begin{remark}
Sometimes we call by a vector-colouring of rank $r$ a mapping $\Lambda\colon \{F_1,\dots,F_m\}\to \mathbb Z_2^s$, 
$s\geqslant r$, satisfying (*) such that $\dim\,\langle\Lambda_1,\dots,\Lambda_s\rangle=r$. 
\end{remark}
\begin{remark}
This notion is not new. It was implicitly used, for example, in \cite{MV86, V87, DJ91, E08}. The analogy between mappings  $\{F_1,\dots,F_m\}\to \mathbb Z_2^r$ and colourings of facets of polytopes was considered in \cite{A10}. In particular, in \cite{A10}
there is a notion of a {\it linearly independent colouring}, which is equivalent to our vector-colouring. For $r=n$ this notion was used already in \cite{NN05}. If all the vectors $\Lambda_1$, $\dots$, $\Lambda_m$ belong to the basis $\{e_1,\dots,e_r\}$ of $\mathbb Z_2^r$, the manifolds defined by such colourings were studied in \cite{I01}.
\end{remark}
Denote by $M(P,\Lambda)$ the orbit space $\mathbb R\mathcal{Z}_P/H(\Lambda)$ of a free action corresponding to a
vector-colouring of rank $r$.  If we identify $\mathbb Z_2^m/{\rm Ker}\,\Lambda$ with $\mathbb Z_2^r$ via the mapping $\Lambda$, then 
$$
M(P,\Lambda)=P\times \mathbb Z_2^r/\sim, \text{ where }(p,a)\sim(q,b)\text{ if and only if }p=q \text{ and }a-b\in\langle \Lambda_i\colon p\in F_i\rangle.
$$ 
In particular, $M(P,\Lambda)$ has a structure of a piecewise linear manifold glued from $2^r$ copies of $P$. It has
an action of $\mathbb Z_2^r$ such that the orbit space is $P$.
\begin{definition}
We call $M(P,\Lambda)$ {\it a manifold defined by a vector-colouring} $\Lambda$.
\end{definition}
\begin{example}
For $r=m$ and the identical mapping $\Lambda_i=e_i$ the manifold $M(P,E)$ is $\mathbb R\mathcal{Z}_P$. 

For $r=n$ a mapping $\Lambda$ satisfying condition (*) is called a {\it characteristic mapping}, and 
the manifold $M(P,\Lambda)$ is called a {\it small cover}.
\end{example}

\begin{proposition}
For vector-colourings $\Lambda_1$ and $\Lambda_2$ of ranks $r_1$ and $r_2$ of a polytope $P$ we have $H(\Lambda_1)\subset H(\Lambda_2)$ if and only if there is a surjection $\Pi\colon\mathbb Z_2^{r_1}\to \mathbb Z_2^{r_2}$ 
such that $\Pi \circ \Lambda_1=\Lambda_2$.
\end{proposition}
\begin{proof}
We have $H(\Lambda_1)\subset H(\Lambda_2)$ if and only if each row of the matrix $\Lambda_2$ with columns $\Lambda_{2,i}$ is a linear combination of rows of $\Lambda_1$. This is equivalent to the existence of a surjection $\Pi\colon\mathbb Z_2^{r_1}\to \mathbb Z_2^{r_2}$ such that 
$\Pi(\Lambda_{1,i})=\Lambda_{2,i}$ for all $i=1$, $\dots$, $m$.
\end{proof}
\begin{corollary}
We have $H(\Lambda_1)=H(\Lambda_2)$ if and only if there is a linear isomorphism $\Pi\colon \mathbb Z_2^{r_2}\to\mathbb Z_2^{r_1}$ such that $\Lambda_1=\Pi\circ\Lambda_2$.
\end{corollary}
If $H(\Lambda_1)\subset H(\Lambda_2)$, then $M(P,\Lambda_2)$ is an orbit space of a freely acting group 
$H(\Lambda_2)/H(\Lambda_1)$ on $M(P,\Lambda_1)$, in particular there is a covering  $M(P,\Lambda_1)\to M(P,\Lambda_2)$
with the fiber $H(\Lambda_2)/H(\Lambda_1)$. 

The following result is a generalization of the result from  \cite{NN05}, which considers 
only small covers.
\begin{proposition}\label{orprop}
Let the vectors $\Lambda_{j_1}, \dots, \Lambda_{j_r}$ form a basis in $\mathbb Z_2^r$. 
Then the  manifold $M(P,\Lambda)$
is orientable if and only if any $\Lambda_i$ is a linear combination of an odd number of these vectors.
\end{proposition}
\begin{proof}
For $M(P,\Lambda)=P\times \mathbb Z_2^r/\sim$ to be orientable it is necessary and sufficient that for any facet $F_i$ of an oriented polytope $P$ 
the polytope $P\times (a+\Lambda_i)$, which is glued to $P\times a$ along this facet, has an opposite orientation. 
Starting from $P\times a$ and using only facets $F_{j_1}$, $\dots$, $F_{j_r}$ 
we can come from $P\times a$ to any $P\times b$, $b\in \mathbb Z_2^r$, which defines uniquely the orientation of any 
polytope $P\times b$. For these orientations to be consistent it is necessary and sufficient that for any facet $F_i$ the polytope
$P\times (a+\Lambda_i)$ is achieved in odd number of steps, which is equivalent to the fact that $\Lambda_i$ is a linear combination of an odd number of vectors $\Lambda_{j_l}$.
\end{proof}

For $3$-polytopes the $4$-colour theorem implies the existence of a small cover over any simple $3$-polytope.
Namely, if all the facets of a $3$-polytope $P$ are coloured in $4$ colours in such a way that adjacent facets have different
colours, then assign to the first three colours the basis vectors $e_1$, $e_2$, $e_3\in\mathbb Z_3$, and to the fours colour
the vector $e_1+e_2+e_3$. Since any three of these four vectors are linearly independent we obtain a characteristic function. 

\begin{corollary}[\cite{NN05}]
A $3$-dimensional small cover $M(P,\Lambda)$ is orientable if and only if its characteristic function corresponds to a colouring
in at most $4$ colours. 
\end{corollary}
\begin{remark}
For bounded right-angled hyperbolic polytopes the orientability of the manifolds corresponding to colourings is mentioned already in \cite{MV86, V87}. 
\end{remark}
  
The maximal dimension of subgroups in $\mathbb Z_2^m$ acting freely on $\mathbb R\mathcal{Z}_P$ is called a real Buchstaber
number $s_{\mathbb R}(P)$ (see \cite{E08, A10, E14, A16}). It is easy to see that $1\leqslant s_{\mathbb R}(P)\leqslant m-n$. The above arguments imply that
$s_{\mathbb R}(P)=3$ for any simple  $3$-polytope.  

Consider all the subgroups in $\mathbb Z_2^m$ acting freely on $\mathbb R\mathcal{Z}_P$. They form a partially ordered set $\mathcal{F}(P)$ with respect to inclusion. After a discussion of this article with Victor Buchstaber
the following natural questions arose. 
\begin{question}
To describe all maximal elements in $\mathcal{F}(P)$.
\end{question}
\begin{question}
To find the set (or the multiset) of integer numbers consisting of ranks of subgroups lying in $\mathcal{F}(P)$.
\end{question}
\begin{question}
To find the minimal rank $m(P)$ of subgroups that are maximal elements of $\mathcal{F}(P)$.
\end{question}
The answer to these questions leads to  new combinatorial invariants generalizing the real Buchstaber number, which equals
to the maximum of the elements of the above (multi)set. 
The following example shows that this answer in general is nontrivial. In particular, $m(P)$ may be less than 
$s_{\mathbb R}(P)$. 

\begin{example}[An orientable manifold corresponding to a maximal subgroup of dim $<m-n$]\label{exas}
On Fig. \ref{exasfig} we present a Schlegel diagram of a $3$-dimensional associahedron $As^3$. 
\begin{figure}
\begin{center}
\includegraphics[scale=0.3]{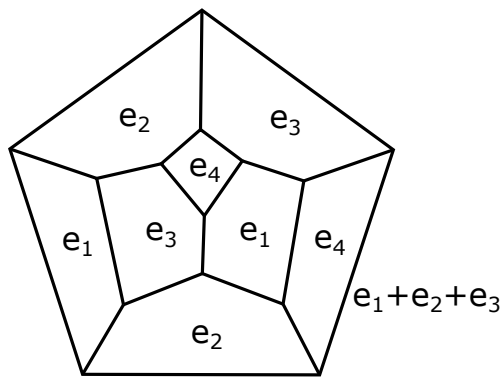}
\end{center}
\caption{The $3$-dimensional associahedron with a vector-colouring corresponding to a maximal (with respect to inclusion) 
freely acting subgroup
of dimension $m-4$}\label{exasfig}
\end{figure}
This polytope can be realized as a
cube $I^3$ with three disjoint pairwise orthogonal edges cut. A Schlegel diagram is a polytopal complex arising when we   
project the boundary complex of the  polytope to one of its facets from a point lying outside the polytope close to this facet. Also 
we present a vector-colouring $\Lambda$ of rank $4$. Proposition \ref{orprop} implies that $M(As^3,\Lambda)$ 
is an orientable manifold
glued of $16$ copies of $As^3$. Assume that there is a surjection $\Pi\colon \mathbb Z_2^4\to\mathbb Z_2^3$ such that
$\Pi\Lambda$ again satisfies condition (*). Since $\Pi(e_1)$, $\Pi(e_2)$, and $\Pi(e_3)$ arise at the same vertex of $As^3$,
these vectors form a basis in $\mathbb Z_2^3$. On the other hand, $\Pi(e_4)$ should be different from
$\Pi(e_1)$, $\Pi(e_2)$, $\Pi(e_3)$, and $\Pi(e_1)+\Pi(e_2)+\Pi(e_3)$, since these vectors correspond to facets adjacent
to  quadrangles corresponding to $\Pi(e_4)$. Also $\Pi(e_4)$ is different from $\Pi(e_1)+\Pi(e_2)$, $\Pi(e_2)+\Pi(e_3)$, and
$\Pi(e_3)+\Pi(e_1)$, since these vectors are sums of vectors corresponding to facets having a common vertex with these quadrangles. Also $\Pi(e_4)\ne 0$ by (*). Thus, there is no vector in $\mathbb Z_2^3$ suitable for being the image of $e_4$.
Therefore, $H(\Lambda)$ is a maximal freely acting subgroup of $\mathbb Z_2^m$ for $As^3$. 
\end{example}

\begin{remark}
For a simple $3$-polytope $P$ the maximal elements in $\mathcal{F}(P)$ are important due to the following fact 
(see \cite[Theorem 1.9.3]{AFW15}). {\it Let $N$ be a compact closed orientable irreducible $3$-manifold, 
$p\colon \widehat{N}\to N$ be a finite cover, and $S_1,\dots S_m$ be tori and Klein bottles from the geometric decomposition of $N$ from Theorem \ref{GDT}. Then $\widehat{N}$ is irreducible and the connected components of $p^{-1}(S_1\sqcup\dots\sqcup S_m)$ give the geometric decomposition for $\widehat{N}$.} In our case each manifold $M(P,\Lambda)$ is covered by $\mathbb R\mathcal{Z}_P$ and covers $M(P,\Lambda')$, where $H(\Lambda')$ is a maximal element in $\mathcal{F}(P)$.
Actually, we will not use the latter covering, since the case of $\mathbb R\mathcal{Z}_P$ is easier, and it is more convenient
to study the covering $\mathbb R\mathcal{Z}_P\to M(P,\Lambda)$. 
\end{remark}
\begin{proposition}
Let $\Lambda_{i_1}$, $\dots$, $\Lambda_{i_r}$ be a basis in $\mathbb Z_2^r$ for a vector-colouring $\Lambda$ of rank $r$
of a simple $n$-polytope $P$. The subgroup $H(\Lambda)$ corresponds to a maximal element in $\mathcal{F}(P)$ if and only if 
for any $j=1,\dots,r$, and any mapping $\Pi_{j,a}\colon \mathbb Z_2^r\to \mathbb Z_2^{r-1}$  defined as 
$$
\Lambda_{i_1}\to e_1,\dots,\Lambda_{i_{j-1}}\to e_{j-1},\Lambda_{i_{j+1}}\to e_j,\dots, \Lambda_{i_r}\to e_{r-1},
\Lambda_{i_j}\to a_1e_1+\dots+a_{r-1} e_{r-1},
$$ 
where $a=(a_1,\dots, a_{r-1})\in \mathbb Z_2^{r-1}$, the composition $\Pi\circ \Lambda$ does not satisfy condition (*).
\end{proposition}
\begin{proof}
If the composition $\Pi\circ \Lambda$ satisfies condition (*), then $H(\Lambda)\subset H(\Pi\circ \Lambda)$ is not maximal.
On the other hand, if $H(\Lambda)$ is not maximal in $\mathcal{F}(P)$,
then it lies in a freely acting subgroup of dimension $\dim H(\Lambda)+1$.
This corresponds to a surjection $\varphi\colon \mathbb Z_2^r\to\mathbb Z_2^{r-1}$ such that $\varphi\circ\Lambda$ satisfies condition (*). Then $r-1$ of the vectors $\varphi(\Lambda_{i_1})$, $\dots$, $\varphi (\Lambda_{i_r})$ form a basis in $\mathbb Z_2^{r-1}$, and there is an automorphism $A$ of $\mathbb Z_2^{r-1}$ such that $A\circ \varphi=\Pi_{j,a}$ for some $j,a$. Then $\Pi_{j,a}$ satisfies condition (*), which is a contradiction.
\end{proof}

\begin{construction}[An orientable double cover]\label{excor} 
Let $M(P,\Lambda)$ be a non-orientable manifold and $\Lambda_{j_1}$,
$\dots$, $\Lambda_{j_r}$ be a basis in $\mathbb Z_2^r$. Below we consider coordinates in this basis.
We have an inclusion $\mathbb Z_2^r\subset \mathbb Z_2^{r+1}$ as a subset
of elements with the last coordinate equal to zero.
Then $\Lambda_{j_1}$, $\dots$, $\Lambda_{j_r}$, $e_{r+1}=(0,\dots,0,1)$ is a basis in $\mathbb Z_2^{r+1}$, 
and we have a surjection
$\Pi\colon \mathbb Z_2^{r+1}\to\mathbb Z_2^r$ given by $\Pi(\Lambda_{j_s})=\Lambda_{j_s}$, $s=1,\dots,r$, 
$\Pi(e_{r+1})=\Lambda_{i_0}$, where $\Lambda_{i_0}$ has an even number of nonzero 
coordinates. Then ${\rm Ker}\, \Pi=\langle e_{r+1}+\Lambda_{i_0}\rangle$. Consider a mapping $\widehat{\Lambda}\colon \{F_1,\dots,F_m\}\to \mathbb Z_2^{r+1}$:
$$
\widehat{\Lambda}_k=
\begin{cases}
\Lambda_k,&\text{if $\Lambda_k$ has an odd number of coordinates},\\
\Lambda_k+(e_{r+1}+\Lambda_{i_0}),&\text{if $\Lambda_k$ has an even number of coordinates.}
\end{cases}
$$ 
\end{construction}
\begin{proposition}\label{doubleprop}
For a non-orientable manifold $M(P,\Lambda)$ the mapping $\widehat{\Lambda}$ is a vector-colouring and $M(P,\widehat{\Lambda})$ is an orientable double cover.
\end{proposition}
\begin{proof}
Let us prove that $\widehat{\Lambda}$ satisfies condition (*). Indeed, if $F_{i_1}\cap\dots\cap F_{i_n}$ is a vertex, then
$\Lambda_{i_1}$, $\dots$, $\Lambda_{i_n}$, $e_{r+1}+\Lambda_{i_0}$ are linearly independent. Assume that 
there is a nontrivial linear dependence  $\mu_1\widehat{\Lambda}_{i_1}+\dots+\mu_n\widehat{\Lambda}_{i_n}=0$.
We have
$$
\mu_1\widehat{\Lambda}_{i_1}+\dots+\mu_n\widehat{\Lambda}_{i_n}=\mu_1\Lambda_{i_1}+\dots+\mu_n \Lambda_{i_n}+
(\mu_{i_{p_1}}+\dots+\mu_{i_{p_l}})(e_{r+1}+\Lambda_{i_0}),
$$
where $p_1$, $\dots$, $p_l\in\{i_1,\dots, i_n\}$ are subscripts of vectors with even numbers of nonzero coordinates. In particular,
$\mu_1=\dots=\mu_n=0$, which is a contradiction. 

We claim that in the basis $\widehat{\Lambda}_{j_1}=\Lambda_{j_1}$, $\dots$, $\widehat{\Lambda}_{j_r}=\Lambda_{j_r}$, 
$\widehat{\Lambda}_{i_0}=e_{r+1}$ each vector $\widehat{\Lambda}_i$
has an odd number of nonzero coordinates. Indeed, if $\Lambda_i$ has an odd number of nonzero coordinates in the basis
$\Lambda_{j_1}$, $\dots$, $\Lambda_{j_r}$, then $\widehat{\Lambda}_i=\Lambda_i$ and and the claim is proved for this vector.
If $\Lambda_i$ has an even number of nonzero coordinates in the basis
$\Lambda_{j_1}$, $\dots$, $\Lambda_{j_r}$, then 
$$
\widehat{\Lambda}_i=\Lambda_i+(e_{r+1}+\Lambda_{i_0})=(\Lambda_i+\Lambda_{i_0})+e_{r+1}
$$
This vector has an even  number of nonzero coordinates corresponding to the first $r$ basis vectors and one -- to $e_{r+1}$.
The claim is proved.  Thus, $M(P,\widehat{\Lambda})$ is an orientable manifold.

Also $\Pi(\widehat{\Lambda}_i)=\Lambda_i$ for all $i$ by construction. Thus, $H(\widehat{\Lambda})\subset H(\Lambda)$ is a subgroup of index $2$, and $M(P,\Lambda)$ is a factor space of a freely acting involution on $M(P,\widehat{\Lambda})$.
This involution changes the orientation. For otherwise, $M(P,\Lambda)$ is orientable.  For each polytope in the subdivision 
of $M(P,\Lambda)$ the manifold $M(P,\widehat{\Lambda})$ has two polytopes with opposite orientations mapping to it under
covering mapping. This is exactly the construction of the orientable double cover. 
\end{proof}

\section{Belts and operations on polytopes}\label{Sec:bo}
\subsection{Belts and families of polytopes}\label{bfam}
\begin{definition}
A {\it $k$-belt} is a cyclic sequence of $k$ facets with the property that facets are adjacent if and only if they follow each other, and no three facets have a common vertex.  A $k$-belt is {\it trivial}, if it surrounds a facet. 
\end{definition}
By the piecewise-linear Jordan curve and Schoenflies theorems each piecewise-linear simple closed curve $\gamma$ 
on the surface  of a $3$-polytope $P$ divides the sphere $\partial P$ into two connected components 
(that is, $\partial P\setminus\gamma=C_1\sqcup C_2$) such that for each component $C_i$ its boundary is $\gamma$, 
and there is a piecewise-linear homeomorphism of the closure $\overline{C_i}$ 
to some polygon mapping $\gamma$ to its boundary. 

For each $k$-belt $\mathcal{B}$ the set $|\mathcal{B}|=\cup_{F_i\in \mathcal{B}}F_i$ is topologically a cylinder
bounded by two simple edge-cycles bounding topological disks outside the belt.  

\begin{definition}
Let us associate to any $k$-belt $\mathcal{B}$ of a simple $3$-polytope $P$ a piecewise-linear simple closed 
curve $\gamma(\mathcal{B})\subset{\rm int}|\mathcal{B}|\subset\partial P$ 
consisting of segments connecting midpoints of edges
of intersection of each facet of the belt with adjacent facets of the belt. We will call $\gamma(\mathcal{B})$ a {\it middle line} of 
$\mathcal{B}$.
\end{definition}

\begin{definition}
A simple polytope $P$ is called {\it flag}, if any set of its pairwise intersecting facets $F_{i_1},\dots,F_{i_k}$ has a 
non-empty intersection $F_{i_1}\cap\dots\cap F_{i_k}\ne \varnothing$. 
\end{definition}

It can be shown that a simple $3$-polytope is flag if and only if it is different from the simplex $\Delta^3$ and 
has no $3$-belts. Also it can be shown that any flag simple $3$-polytope has $m\geqslant6$ facets, and if it has $6$ facets, then
it is combinatorially equivalent to the cube $I^3$.

Results by  A.V.~Pogorelov  \cite{P67} and E.M.~Andreev \cite{A70a} imply that a simple $3$-polytope $P$
can be realized in the Lobachevsky (hyperbolic) space $\mathbb H^3$ as a bounded polytope with right dihedral angles
if and only if $P$ is different from the simplex $\Delta^3$ and has no $3$- and $4$-belts. Moreover, a realization is unique up to 
isometries. Such polytopes are called
{\it Pogorelov polytopes}. As mentioned in \cite{DJ91} small covers over bounded right-angled polytopes
have a natural hyperbolic structure. The same manifolds were discovered by A.D.\,Mednykh and Yu.\,Vesnin \cite{MV86,V87} (see Construction \ref{MVconstr}) as generalizations of a manifold invented by F.~L\"obell.
Results by G.D.~Birkhoff \cite{B1913} imply that the $4$-colour problem 
can be reduced to colouring facets of Pogorelov polytopes with only trivial $5$-belts.

\begin{corollary}
The four colour theorem is equivalent to the fact that any bounded right-angled hyperbolic $3$-polytope with
only trivial $5$-belts admits an orientable small cover.
\end{corollary}


In this paper an important role is played by the family of {\it almost Pogorelov polytopes}. It consists of simple $3$-polytopes 
$P\ne \Delta^3$ without $3$-belts such that any $4$-belt is trivial. The combinatorics and hyperbolic geometry of this
family of polytopes was studied in \cite{E19}.
 
The following fact can be extracted directly from the
definition. 
\begin{proposition}
The cube $I^3$ (the $4$-prism) and the $5$-prism are the only almost Pogorelov polytopes with adjacent quadrangles.
All the other almost Pogorelov polytopes have no adjacent quadrangles and $m\geqslant 9$ facets.
\end{proposition}

Results by E.~M.~Andreev \cite{A70a, A70b} imply that almost Pogorelov 
polytopes correspond to right-angled  polytopes of finite volume in  $\mathbb H^3$ 
(right-angled means that all dihedral angles between facets intersecting by edges, perhaps with vertices on the absolute,  
are $\frac{\pi}{2}$). 
Such polytopes may have $4$-valent vertices on the absolute, while all proper vertices have valency $3$.
Two right-angled polytopes of finite volume in $\mathbb H^3$ are congruent if and only if they are combinatorially equivalent
(see \cite[Ch. 5, Sect. 2.1]{VS88}). 
\begin{proposition}\cite[Theorem 10.3.1]{DO01} (see also \cite[Theorem 6.5]{E19})
Cutting off $4$-valent vertices defines a bijection between classes of congruence of right-angled
polytopes of finite volume in $\mathbb{H}^3$ and almost Pogorelov polytopes without adjacent quadrangles. Moreover, it induces a bijection between the ideal vertices of the right-angled polytope and the 
quadrangles of the corresponding almost Pogorelov polytope.
\end{proposition} 
The $3$-dimensional associahedron $As^3$ is a unique almost Pogorelov polytope with $m=9$. 
It corresponds to a right-angled $3$-gonal bipyramid. The bipyramid has 
two proper vertices of valency $3$ and three ideal vertices of valency $4$. 

For recent results on volumes of right-angled hyperbolic polytopes see \cite{VE20, EV20b,EV20c}. 
In \cite{LW20} the right-angled hyperbolicity of handlebodies with simple facial structure is studied.

\subsection{Nested families of belts and curves}
In this subsection we develop a technique allowing to work with families of belts. 
\begin{definition}
Let $P$ be a simple $3$-polytope. We 
call two belts $\mathcal{B}_1$ and $\mathcal{B}_2$ {\it compatible}, if the middle line $\gamma(\mathcal{B}_2)$ 
lies in the closure of a connected component of $\partial P\setminus\gamma(\mathcal{B}_1)$. 
\end{definition}
\begin{remark}
This definition can be equivalently reformulated in terms of dual simplicial polytope $P^*$.
A belt $\mathcal{B}$ of $P$ corresponds to the chordless cycle $c(\mathcal{B})$ in $P^*$ that does not bound a triangle.
Belts $\mathcal{B}_1$ and $\mathcal{B}_2$ are  compatible if and only if $c(\mathcal{B}_2)$ lies in the closure 
of a connected component of $\partial P^*\setminus c(\mathcal{B}_1)$. Since $P$ and $P^*$ 
have combinatorially equivalent barycentric subdivisions, $\partial P^*$ can be combinatorially realized 
in the barycentric subdivision of $\partial P$: the vertex $\{i\}$ of $P^*$ corresponds to the barycentre of the facet $F_i$ of $P$, 
the edge $\{i,j\}$ -- to the curve consisting of two segments connecting the centre of the edge $F_i\cap F_j$ with $\{i\}$ and $\{j\}$,
and the triangle $\{i,j,k\}$ -- to the area consisting of $6$ triangles. 
Then $\gamma(\mathcal{B})$ is isotopic to $c(\mathcal{B})$ in this realization.
\end{remark}
\begin{lemma}
Two belts $\mathcal{B}_1$ and $\mathcal{B}_2$  are compatible if and only if a component of $\partial P\setminus\gamma(\mathcal{B}_2)$ lies in a component of  $\partial P\setminus\gamma(\mathcal{B}_1)$.
\end{lemma}
\begin{proof}
Let $\gamma_1=\gamma(\mathcal{B}_1)$ divide $\partial P$ into components $C_1$ and $C_2$, 
$\gamma_2=\gamma(\mathcal{B}_2)$ -- into $D_1$ and $D_2$, and $\gamma_2\subset\overline{C_1}$. 
The component $C_2$ is path-connected and  does not intersect $\gamma_2$, therefore $C_2$ lies in $D_1$ or $D_2$, 
say in $D_2$. Then $D_1\cap \overline{C_2}=\varnothing$. In particular, $D_1\subset C_1$.

On the other hand, if $D_1\subset C_1$, then $\gamma_2=\partial D_1\subset \overline{C_1}$.
\end{proof}

Following the notations of the proof, if $D_1\subset C_1$, then $C_2\cap \overline {D_1}=\varnothing$. In particular,
$C_2\subset D_2$. Thus, the definition is symmetric with respect to $\gamma_1$ and $\gamma_2$.

\begin{lemma}\label{ncb-lemma}
If the belts $\mathcal{B}_1$ and $\mathcal{B}_2$ are not compatible, then each of them contains two facets lying in the closures
of different connected components of the complement to the other belt~in~$\partial P$.
\end{lemma}
\begin{proof}
Let $\gamma_1=\gamma(\mathcal{B}_1)$ divide $\partial P$ into connected components $C_1$ and $C_2$, and $\gamma_2=\gamma(\mathcal{B}_2)$.  If the belts are not compatible, 
then $\gamma_2$ has points in $C_1$ and $C_2$. Each segment of $\gamma_2$ connects two midpoints of non-adjacent edges 
of a facet $F_i$ of $\mathcal{B}_2$. Let such a segment $E_2$ have 
a point in $C_1$. If  $F_i$ does not belong to $\mathcal{B}_1$,
then $F_i\subset C_1$. Else, there is a segment $E_1$ of $\gamma_1$ in $F_i$. Since the segments  do not coincide
(for otherwise, $E_2\subset \gamma_1$ can not have points in $C_1$), either $E_2$ does not intersect $E_1$, or it intersects  
$E_1$ transversally, or they have a unique common vertex. In all the cases $E_2$ has a vertex in $C_1$. This vertex is a midpoint
of the edge of intersection of $F_i$ with some facet $F_j\in \mathcal{B}_2$.
The facet $F_j$ has a point in $C_1$, is adjacent to $F_i$ and is not successive to it in $\mathcal{B}_1$. Therefore, $F_j\notin \mathcal{B}_1$, and $F_j\subset C_1$. Similarly one can find a facet in $\mathcal{B}_2$ lying in $C_2$. This finishes the proof. 
\end{proof}

\begin{definition}\label{NSB}
Let us call by a {\it nested family} of belts a collection of belts such that any two of them are compatible.
\end{definition}

\begin{corollary}\label{3bnested}
For any simple $3$-polytope its $3$-belts form a nested family (possibly empty).
\end{corollary}
\begin{proof}
If the $3$-belts $\mathcal{B}_1$ and $\mathcal{B}_2$ are not compatible, then 
$\mathcal{B}_1$ has two facets in different connected components of $\partial P\setminus\mathcal{B}_2$.
Hence, these facets do not intersect. A contradiction.
\end{proof} 

\begin{definition}\label{def:nsc}
Let us call by a {\it nested family of curves} a family of piecewise linear simple curves $\widetilde{\gamma}(\mathcal{B})$  bijectively corresponding to the belts $\mathcal{B}$ from a nested family and satisfying the following conditions:
\begin{enumerate}
\item segments of each curve $\widetilde{\gamma}(\mathcal{B})$ bijectively correspond to facets 
of the belt $\mathcal{B}$ and each segment connects two points in the relative interiors of the edges of intersection of the 
corresponding facet with the adjacent facets of the belt;
\item the curves are pairwise disjoint;
\end{enumerate}
\end{definition}

For any vertex $v$ of an edge $F_i\cap F_j\in P$ and a belt $\mathcal{B}$ containing $F_i$ and $F_j$,
there is a unique connected component $C_v(\mathcal{B})$ of $\partial P\setminus\gamma(\mathcal{B})$ containing $v$.

\begin{proposition}\label{vorder}
For any nested family of curves and any edge $E$ of $P$ containing both vertices $v_p\in\widetilde{\gamma}(\mathcal{B}_p)$ and 
$v_q\in\widetilde{\gamma}(\mathcal{B}_q)$ the vertex $v$ of $E$ is closer $v_p$ than to $v_q$ if and only if 
$C_v(\mathcal{B}_p)\subset C_v(\mathcal{B}_q)$.
\end{proposition}
\begin{proof}
Assume that $C_v(\mathcal{B}_p)\subset C_v(\mathcal{B}_q)$ and $v$ is closer to $v_q$. Consider the segments of $\widetilde{\gamma}(\mathcal{B}_p)$ and $\widetilde{\gamma}(\mathcal{B}_q)$ lying on the same facet and starting at $v_p$ and $v_q$. If their other ends $v_p'$ and $v_q'$ lie on the same edge, then for this edge again its vertex 
$v'\in C_{v'}(\mathcal{B}_p)=C_v(\mathcal{B}_p)\subset C_v(\mathcal{B}_q)=C_{v'}(\mathcal{B}_q)$ is closer to $v'_q$ than to $v_p'$,  since the curves are disjoint. Then take the next segments. Since the curves correspond to different compatible belts, after several steps we come to two segments at one facet starting at the same edge and finishing at different edges of this facet. Moreover, for the starting edge again its vertex  
$\widehat{v}\in C_{\widehat{v}}(\mathcal{B}_p)\subset C_{\widehat{v}}(\mathcal{B}_q)$ 
is closer to the vertex of  $\widetilde{\gamma}(\mathcal{B}_q)$. The ending edge for the segment of 
$\widetilde{\gamma}(\mathcal{B}_p)$ lies in $C_{\widehat{v}}(\mathcal{B}_q)$, and the ending edge  for the segment of 
$\widetilde{\gamma}(\mathcal{B}_q)$  lies in $\partial P\setminus \overline{C_{\widehat{v}}(\mathcal{B}_p)}$. 
Then the segments should intersect. A contradiction.
\end{proof}

\begin{lemma}\label{NFC-lemma}
Any nested family of belts admits a nested family of curves.
\end{lemma}
\begin{proof}
We start with a definition.
\begin{definition}
Let us call a nested family of belts {\it cylindrical}, if there is an ordering of belts $\mathcal{B}_1$, \dots, $\mathcal{B}_N$
and a choice of connected components $C(\mathcal{B}_i)$ of $\partial P\setminus \gamma(\mathcal{B}_i)$ 
such that $C(\mathcal{B}_i)\subset C(\mathcal{B}_{i+1})$ for $i=1,\dots,N-1$.
\end{definition}

\begin{lemma}\label{e-c-lemma}
For any nested family of belts and an edge $F_i\cap F_j$ the subfamily of belts containing $F_i$ and $F_j$ is cylindrical. 
\end{lemma} 
\begin{proof}
A belt $\mathcal{B}$ contains $F_i$ and $F_j$ if and only if $\gamma(\mathcal{B})$ contains the midpoint of $F_i\cap F_j$.
Moreover, the vertices $v$ and $w$ of the edge $F_i\cap F_j$ lie in different connected 
components of  $\partial P\setminus\gamma(\mathcal{B})$.  Since $v\notin C_w(\mathcal{B})$, we have $C_v(\mathcal{B})\not\subset C_w(\mathcal{B}')$ for any two belts $\mathcal{B}$ and $\mathcal{B}'$ containing $F_i$ and $F_j$. 
Since the belts are compatible,
$C_v(\mathcal{B})\subset C_v(\mathcal{B}')$ or $C_v(\mathcal{B}')\subset C_v(\mathcal{B})$. 
Then the ordering of the belts we need is 
the inclusion ordering of the components  $C_v(\mathcal{B})$.
\end{proof}

\begin{construction}\label{con:hgamma}
For a nested family of belts and each edge  $F_i\cap F_j$ of $P$ take its vertex $v$ and
consider the ordering of the belts containing $F_i$ and $F_j$ by inclusion of components $C_v(\mathcal{B})$.
Put in the interior of $F_i\cap F_j$ in the same order in direction from $v$ to $w$ disjoint points corresponding to these belts 
(for example, divide the edge $F_i\cap F_j$ into equal parts by these points).
Then for each belt $\mathcal{B}$ we change the middle line $\gamma(\mathcal{B})$ to the piecewise-linear curve 
$\widehat{\gamma}(\mathcal{B})$ consisting of segments in facets of the belt connecting the assigned points on the edges 
of intersection with adjacent  facets of the belt. 
\end{construction}

\begin{lemma}\label{lem:disj}
For any two belts $\mathcal{B}_1$ and $\mathcal{B}_2$ from the nested family of belts the curves
$\widehat{\gamma}(\mathcal{B}_1)$ and $\widehat{\gamma}(\mathcal{B}_2)$ do not intersect.
\end{lemma}
\begin{proof}
Let the curves $\widehat{\gamma}(\mathcal{B}_1)$ and $\widehat{\gamma}(\mathcal{B}_2)$ have a common point.
By construction these curves have no common points on the edges and do not pass through the vertices of a polytope.
Then their common point is an interior point of some facet $F_i$. In particular, this point is the intersection of some 
segments $\widehat{E_1}$ and $\widehat{E_2}$ of these curves lying in $F_i$.
If ends of these segments lie on four different edges, then the corresponding curves $\gamma(\mathcal{B}_1)$ and 
$\gamma(\mathcal{B}_2)$ intersect transversally, which is a contradiction. Thus, ends of the segments lie either
on three or two edges of $F_i$. For  $\gamma(\mathcal{B}_1)$ and $\gamma(\mathcal{B}_2)$ the corresponding segments
$E_1$ and $E_2$ ether have a unique common vertex, or coincide. Without loss of generality assume that for the edge of $P$ 
containing the vertex $v$ of $E_1\cap E_2$ we have  $C(\mathcal{B}_1)\subset C(\mathcal{B}_2)$. In the first case the other vertex of $E_1$ lies in $C(\mathcal{B}_2)$. 
After the shift the first end of $\widehat{E_1}$ lies closer to $v$ than the corresponding end of $\widehat{E_2}$, while the other 
ends of these segments stay at the same edges as before. Hence, $\widehat{E_1}\cap \widehat{E_2}=\varnothing$. In the second case for both edges of $F_i$ containing vertices of $E_1$ and $E_2$
their vertices $v$ and $v'$ lie in $C(\mathcal{B}_1)\subset C(\mathcal{B}_2)$. Therefore 
after the shift the edges $\widehat{E_1}$ and $\widehat{E_2}$ do not intersect.
\end{proof}
This finishes the proof of Lemma \ref{NFC-lemma}.
\end{proof}

\begin{lemma}\label{lem:iso}
Any two nested families of curves corresponding to the same nested family of belts are isotopic in the class of nested families of curves.
\end{lemma}
\begin{proof}
Indeed, the curves are defined uniquely by a sequence of points at each edge of the polytope. By Proposition \ref{vorder}
on each edge the orders of the points corresponding to different belts are the same. So 
we can move one sequence of points to the other by an isotopy at each edge. 
\end{proof}

\subsection{Operations on polytopes}

\begin{construction}\label{con:bcut}
For any $k$-belt of a simple $3$-polytope one can define operation of {\it cutting the polytope along the belt}. 
Combinatorially take a piecewise linear closed curve consisting of segments 
connecting the interior points of edges of intersection of each facet of a belt with the adjacent facets of the belt, for example 
$\gamma(\mathcal{B})$. Then cut the surface of the polytope along this curve to obtain two topological disks partitioned into polygons. Glue up each disk by a polygon along its boundary.  We obtain two partitions of a sphere into polygons. Using 
the Steinitz theorem it is easy to see that each partition is combinatorially equivalent to the boundary of a simple polytope. 
Denote these simple polytopes $P_1$ and $P_2$.  Each polytope is different from the simplex, since no three facets 
of the belt have a common vertex.
\end{construction}
The facets of the polytope $P_i$ are of three types: 
\begin{enumerate}
\item the facets corresponding to the facets $F_a$ of $P$ lying in the connected component $C_i(\mathcal{B})$ of $\partial P\setminus \gamma(\mathcal{B})$ corresponding to $P_i$. We denote them $\widehat{F_a}$.
\item the facets corresponding to the facets $F_a\in \mathcal{B}$. We denote them also $\widehat{F_a}$.
\item the new facet $F$ corresponding to the polygon glued along the boundary to $\gamma(\mathcal{B})$.
\end{enumerate}
\begin{remark}\label{famcut}
Lemma \ref{lem:iso} implies that an operation of {\it cutting a simple $3$-polytope along a nested family of belts} is well-defined.
\end{remark}
\begin{lemma}\label{intlemma}
The facets $\widehat{F_{a_1}}$, $\dots$, $\widehat{F_{a_k}}$ intersect in $P_i$ if and only if the corresponding
facets $F_{a_1}$, $\dots$, $F_{a_k}$ intersect in $P$.
\end{lemma}
\begin{proof}
If $F_{a_1}$, $\dots$, $F_{a_k}$ intersect in $P$, consider their intersection. It is either a vertex or an edge. 
If the corresponding facets  $\widehat{F_{a_1}}$, $\dots$, $\widehat{F_{a_k}}$ do not intersect, then this vertex or edge
lies outside $\overline{C_i(\mathcal{B})}$. This is possible only if all the facets $F_{a_1}$, $\dots$, $F_{a_k}$ belong to $\mathcal{B}$. Then $k=2$ and $F_{a_1}$ and $F_{a_2}$ are successive in $\mathcal{B}$. Then $F_{a_1}\cap F_{a_2}$ intersects 
$\gamma(\mathcal{B})$, which is a contradiction. On the other hand, the construction of $P_i$ does not create
new intersections of old facets, that is if $F_{a_1}$, $\dots$, $F_{a_k}$ do not intersect in $P$, then $\widehat{F_{a_1}}$, $\dots$, $\widehat{F_{a_k}}$ do not intersect in $P_i$.
\end{proof}
\begin{lemma}\label{Piblemma}
For each polytope $P_i$, $i=1,2$, its $3$-belts bijectively correspond to $3$-belts of $P$ consisting of facets presented in $P_i$.
(Equivalently, $3$-belts $\mathcal{L}$ of $P$ with $\gamma(\mathcal{L})\subset \overline{C_i(\mathcal{B})}$.) This bijection preserves triviality of belts, except for case when the belt around $F$ corresponds to the nontrivial $3$-belt $\mathcal{B}$.
\end{lemma}
\begin{proof}
First let us prove that $3$-belts of $P$ consisting of facets presented in $P_i$ are exactly $3$-belts $\mathcal{L}$ of $P$ with $\gamma(\mathcal{L})\subset \overline{C_i(\mathcal{B})}$. Indeed, if $\mathcal{L}$ consists of facets presented in $P_i$, 
then centres of edges of intersection of its successive facets lie in  $\overline{C_i(\mathcal{B})}$. Therefore, $\gamma(\mathcal{L})\subset \overline{C_i(\mathcal{B})}$.
And visa versa, if $\gamma(\mathcal{L})\subset \overline{C_i(\mathcal{B})}$, then centres of 
edges of intersection of successive facets of $\mathcal{L}$ lie in  $\overline{C_i(\mathcal{B})}$. This implies that all the facets 
of $\mathcal{L}$ are presented in $P_i$.

The facet $F$ does not lie in any $3$-belt. For otherwise, this belt has the form $(F,\widehat{F_a},\widehat{F_b})$,
where $F_a, F_b\in\mathcal{B}$. In this case $F_a\cap F_b\ne\varnothing$. Hence, $F_a$ and $F_b$ are successive facets of 
$\mathcal{B}$. In particular, $F_a\cap F_b$ is the edge with a vertex on $\gamma(\mathcal{B})$, which corresponds to the vertex $F\cap\widehat{F_a}\cap\widehat{F_b}$.

The correspondence  
$\widehat{\mathcal{L}}=(\widehat{F_a},\widehat{F_b},\widehat{F_c})\leftrightarrow (F_a,F_b,F_c)=\mathcal{L}$ 
between $3$-belts
of $P_i$ and $3$-belts of $P$ consisting of facets presented in $P_i$ follows from Lemma \ref{intlemma}.
If the belt $\widehat{\mathcal{L}}$ surrounds a facet $\widehat{F_d}$, then $\mathcal{L}$ surrounds $F_d$.
If $\widehat{\mathcal{L}}$ surrounds the triangle $F$, then this belt corresponds to the nontrivial $3$-belt $\mathcal{B}$.
If the $3$-belt $\mathcal{L}$ consisting of facets presented in $P_i$ surrounds a facet $F_d$ of $P$, 
then $F_d$ is also presented in $P_i$. For otherwise, $\mathcal{L}$ can be only $\mathcal{B}$, which is nontrivial.
Also $F_d$ does not belong to $\mathcal{B}$, since a belt can not contain triangles. 
Thus, $\widehat{\mathcal{L}}$ surrounds $\widehat{F_d}$. 
\end{proof}

\begin{corollary}\label{Pibcor}
Any $3$-belt of $P$ different from  $\mathcal{B}$ corresponds to a $3$-belt of exactly one polytope $P_1$ or $P_2$. 
\end{corollary}
\begin{proof}
Indeed, since $3$-belts form a nested family, any $3$-belt of $P$ either is $\mathcal{B}$, or consists of facets presented 
simultaneously in one of the polytopes $P_1$ and $P_2$. If it consists of facets presented in both $P_1$ and $P_2$, then 
it coincides with $\mathcal{B}$.
\end{proof}

\begin{construction}
Given simple $3$-polytopes $P_1$ and $P_2$ with chosen facets $F_i$ and $F_j$ surrounded by $k$-belts
with the same $k$ one can take a {\it connected sum of $P_1$ and $P_2$ along $k$-gons}. This is a simple $3$-polytope
$P$ such that $\partial P$ is combinatorially obtained  from $\partial P_1$ and $\partial P_2$ by deletion of interiors of chosen facets and gluing the resulting disks partitioned into polygons along $k$-cycles followed by a deletion of this $k$-cycle.
The Steinitz theorem implies that $P$ is indeed a simple $3$-polytope. The operation depends on the identification of boundaries of chosen facets. It is inverse to the operation of cutting a polytope along the belt. 
\end{construction}
\begin{corollary}
Any $k$-belt $\mathcal{B}$ of a simple $3$-polytope $P$ corresponds to a decomposition of $P$ into a connected sum of some polytopes $P_1$ and $P_2$ along $k$-gons.
\end{corollary}
\begin{proof}
This is a corollary of Construction \ref{con:bcut}.
\end{proof}

\begin{definition}
A {\it connected sum} of simple $3$-polytopes $P$ and $Q$ along {\it vertices} $v\in P$, $w\in Q$ is a simple 
$3$-polytope $P\#_{v,w}Q$ which is obtained by cutting off vertices $v$ and $w$ and taking the connected sum along 
the arising triangles. 
\end{definition}
This connected sum operation gives a $3$-belt on the resulting polytope.  

\begin{remark}
The Steinitz theorem implies that for a $3$-belt $\mathcal{B}$ in Construction \ref{con:bcut} each polytope $P_i$ admits a shrinking of the triangle obtained from 
the belt into a point to produce a new $3$-polytope $Q_i$. Then $P$ is a connected sum of $Q_1$ and $Q_2$ along vertices.
Topologically the polytopes $Q_1$ and $Q_2$ are obtained from $P$ by cutting along the middle line of the belt and shrinking the
arising triangles to points.
Thus, each triangle of a polytope $P\ne\Delta^3$ corresponds to a connected sum of a polytope $Q$ with this triangle shrinked with a simplex along vertices. 
\end{remark}
\begin{corollary}\label{cor:3bQ}
Let a simple $3$-polytope $P$ be a connected sum of $Q_1$ and $Q_2$ along vertices. 
Then for each polytope $Q_i$, $i=1,2$, its $3$-belts bijectively correspond to $3$-belts of $P$ different from the arising 
$3$-belt and consisting of facets presented in $Q_i$. 
\end{corollary}
\begin{proof}
This follows from Corollary \ref{Pibcor}. 
Namely, since the facet $F$ can not belong to any $3$-belt of $P_i$, shrinking this facet to a point transforms $3$-belts
of $P_i$ different from the belt around $F$ to $3$-belts of $Q_i$. On the other hand, the polytope $P_i$
is obtained from $Q_i$ by cutting off a vertex. This operation evidently transforms $3$-belts to $3$-belts. 
\end{proof}

Now let us consider a geometric realization of a connected sum of two simple $3$-polytopes along triangles and quadrangles. 
We will need it below. For triangles the result is simpler.
\begin{proposition}\label{3bpt}
Let $P$ and $P'$ be geometric simple $3$-polytopes, different from the simplex, and $F_i$ and $F_j'$ be their triangular
facets. Choose an identification of $F_i$ 
and $F_j'$ as combinatorial polygons. Then up to projective transformations $P$ and $P'$ may
be glued along congruent facets $F_i$ and $F_j'$ following the combinatorial identification 
to produce a geometric $3$-polytope $Q$
having a $3$-belt instead of these facets.
\end{proposition}
\begin{proof}
Consider the three edges of $P$ intersecting $F_i$ at vertices. Each pair of these edges lies in a plane. Therefore, either these edges are parallel, or they intersect at a common point.  In the latter case the point does not belong to $P$. For otherwise, $P=\Delta^3$. Then there is a projective transformation mapping this point to a point at infinity and $P$ to a bounded polytope. In the new polytope the three edges are parallel. By an affine transformation we can make them orthogonal to $F_i$ and this facet 
a regular triangle. Applying similar transformations to $P'$, and, perhaps, a reflection, we can glue the resulting polytopes along congruent facets following the combinatorial identification. 
\end{proof}

\begin{corollary}\label{3bplane}
For any family $\mathcal{F}$ of $3$-belts of a simple $3$-polytope $P$ there is a geometric realization of $P$ and a family of planes bijectively corresponding to belts such that the intersections of planes with $\partial P$ form a nested family of curves corresponding to $\mathcal{F}$.
\end{corollary}
\begin{proof}
Indeed, by Corollary \ref{3bnested} $\mathcal{F}$ is a nested family of belts. By Lemmas \ref{NFC-lemma} and \ref{lem:iso}
this family uniquely defines a decomposition of $P$ into a connected sum of polytopes along triangles. These operations may be encoded by a tree, where vertices correspond to polytopes and edges to pairs of identified triangles. Take any geometric realizations of these polytopes. Then starting from any vertex of the tree as a root and applying projective transformations using Proposition \ref{3bpt} we can glue one by one all the polytopes with respect to combinatorial identifications of their triangles. The resulting polytope is combinatorially equivalent to $P$, the identified triangles correspond to planar sections of $P$, and these sections are disjoint.
\end{proof}

For a connected sum along quadrangles situation is a little bit more complicated.
Let $P$ be a simple geometric $3$-polytope in $\mathbb R^3$, and $F_i$ be its quadrangular facet surrounded by a $4$-belt
$(F_j,F_k,F_l,F_r)$. Then any three of the four planes $\pi_j={\rm aff}(F_j)$, 
$\pi_k={\rm aff}(F_k)$, $\pi_l={\rm aff}(F_l)$, and $\pi_r={\rm aff}(F_r)$ 
do not contain the same line in $\mathbb RP^3$. There are two possibilities: either these planes contain a common point of $\mathbb RP^3$, or not. In the first case up to a projective transformation we can assume that this point lies at infinity. Then these four planes in $\mathbb R^3$ bound an infinite cylinder over $F_i$, which in $\mathbb R P^3$ together with one point 
is divided by $F_i$ into two $4$-pyramids. Let us call 
the pyramid containing $P$ an {\it inner pyramid} associated to $F_i$, and the other pyramid -- an {\it outer pyramid}.
The addition of the outer pyramid to $P$ transforms it to a polytope with combinatorially $F_i$ shrinked to a point.
In the second case the four planes are in general position in $\mathbb RP^3$ and up to a projective transformation 
these planes bound a simplex containing $P$.
The plane $\pi_i={\rm aff} (F_i)$ divides this simplex into two polytopes combinatorially equivalent to $3$-prisms. 
These prisms have a common quadrangle $F_i$ and are "twisted", that is bases of each prism correspond to
quadrangles of the other. Let us call the prism containing $P$ an {\it inner prism}
associated to $F_i$, and the other $3$-prism -- an {\it outer prism}. 
There are two possibilities: $F_j$ and $F_l$ correspond to bases either of the inner prism, 
or of the outer prism.
In the first case addition of the outer prism to the polytope transforms $P$ to a polytope with combinatorially 
$F_i$ shrinked to an edge "parallel" to $F_i\cap F_j$ and $F_i\cap F_l$, and in the second case -- to $F_i\cap F_k$ and 
$F_i\cap F_r$. Each of the three possibilities is a projective invariant of the geometric polytope $P$.

\begin{proposition}\label{4bpt}
Let $P$ and $P'$ be geometric simple $3$-polytopes, and $F_i$ and $F_j'$ be their quadrangular
facets surrounded by belts. Choose an identification of $F_i$ 
and $F_j'$ as combinatorial polygons. Then up to projective transformations $P$ and $P'$ may
be glued along congruent facets $F_i$ and $F_j'$ following the combinatorial identification 
to produce a geometric $3$-polytope $Q$
having a $4$-belt instead of these facets if and only if either both $F_i$ and $F_j'$ correspond to $4$-pyramids, or 
to $3$-prism, and the inner $3$-prisms are "twisted" under the identification of their quadrangles.
\end{proposition}  
\begin{proof}
From the above arguments we see that if after projective transformation we glue the polytopes along
congruent facets, then either both facets $F_i$ and $F_j'$ correspond to $4$-pyramids, and the inner $4$-prism of one polytope 
is identified with the outer $4$-prism of the other, or both facets correspond to $3$-prism, and the inner  
$3$-prism of one polytope is identified with the outer $3$-prism of the other polytope. In particular, the inner $3$-prism are twisted under the identification.
On the other hand, if $F_i$ and $F_j'$ correspond to $4$-pyramids, then there are projective transformations making these facets congruent squares orthogonal to adjacent facets, and we can glue the polytopes following the combinatorial identification of facets, perhaps after a reflection of one of the polytopes.
If $F_i$ and $F_j'$ correspond to twisted $3$-prisms, then using the classical facts that any two combinatorial $3$-prisms are projectively equivalent \cite[Exercise 6.24]{Z95} and the group of symmetries of a right triangular prism with regular bases acts transitively on pairs (a quadrangle, its edge orthogonal to bases), we can find a projective  transformation mapping the inner $3$-prism associated to $F_i$ to the outer $3$-prism associated to $F_j'$ and respecting the combinatorial identification of these facets. This finishes the proof. 
\end{proof}

\begin{corollary}\label{4bplane}
For any nested family $\mathcal{F}$ of $4$-belts of a simple $3$-polytope $P$ there is a geometric realization of $P$ and a family of planes bijectively corresponding to belts such that the intersections of planes with $\partial P$ form a nested family of curves corresponding to $\mathcal{F}$.
\end{corollary}
\begin{proof}
By Lemmas \ref{NFC-lemma} and \ref{lem:iso} the family $\mathcal{F}$ uniquely defines a decomposition of $P$ into 
a connected sum of polytopes along quadrangles. These operations may be encoded by a tree $T$, where vertices correspond 
to polytopes and edges to pairs of identified quadrangles. Choose any vertex of $T$ as a root and take any geometric
realization of the corresponding  polytope $Q$. Since $Q$ is simple, there is an arbitrary small deformation of its 
defining inequalities such that the new realization has the same combinatorial type and all the planes containing 
facets are in general position,  that is there is no point in $\mathbb R^3$ lying in more than three of these planes. 
Then any quadrangle of $Q$ corresponds to its inner $3$-prism. 
Consider the edge of $T$ corresponding to an identification of a quadrangle $F_i$ of 
$Q$ with a quadrangle $F_j'$ of another polytope $R$. According to Proposition \ref{4bpt} to make this identification 
geometric after projective transformations one needs to realize $R$ in such a way that the inner $3$-prism of $R$ is twisted 
under the identification. Again the realization may be deformed to be in general position. Since $F_j'$ is surrounded by a belt, 
the Steinitz theorem implies that both polytopes arising from $R$ by shrinking $F_j'$ to an edge in two possible ways exist. 
Cutting this edge we obtain two realizations of $R$ with the inner $3$-prisms of $F_j'$ twisted under
identification. So one of these realizations fits our condition. Repeating this argument for each new edge of $T$ in the end we 
obtain a polytope combinatorially equivalent to $P$. In this polytope the identified quadrangles correspond to planar sections and these sections are disjoint.
\end{proof}

\section{Prime decomposition for manifolds corresponding to vector-colourings}\label{Sec:prime}
\subsection{General information}

It is important, that in dimension $3$ there is no difference between continuous, smooth and piecewise-linear categories. 
Namely, it is known that  any topological $3$-manifold (not necessarily compact) admits a unique piecewise-linear structure and 
a unique smooth structure. A topological action of
a finite group on a closed $3$-manifold $N$ is smoothable if and only if it is
simplicial in some triangulation of $N$.

\begin{definition}
By a topological $2$-submanifold of a closed topological $3$-manifold $N$
we mean a subset $S$ of $N$ such that for any point $p\in N$ there exists a homeomorphism $j\colon U\to V$ 
from an open neighbourhood of $p$ in $N$ to an open subset of $\mathbb R^3$ 
such that $j(U\cap S)\subset\{0\}\times \mathbb R^2$.
\end{definition}

It is essential that one can replace topological $2$-submanifolds by $PL$-submanifolds and by
smooth submanifolds. In particular, if $N$ is a topological compact $3$-manifold and 
$S$ is a topological $2$-submanifold of $N$, then there exists a smooth (or $PL$) structure on 
$N$ and a smoothly embedded surface $S'$ in $N$ that is isotopic to $S$ (see \cite[Section 1.1]{AFW15}).

\begin{definition}
By $\nu(S)$ we denote a tubular neighbourhood of the submanifold $S$
in $N$. Given a surface $S$ in $N$ we refer to $N\setminus\nu(S)$ as $N$ cut along $S$. 
\end{definition}

In the case of manifolds corresponding to vector-colourings  there are natural structures of piecewise-linear manifolds and submanifolds. 

\begin{definition}
Let $N_1$ and $N_2$ be oriented $3$-manifolds and let $B_i\subset N_i$, $i=1,2$, be embedded closed
$3$-balls with boundaries $2$-submanifolds homeomorphic to $S^2$.
For i = 1,2 we endow $\partial B_i$ with the orientation induced from the orientation
as a boundary component of $N_i\setminus\,{\rm  int} B_i$. We pick an orientation-reversing homeomorphism
$f \colon \partial B_1 \to \partial B_2$ and we refer to
$N_1\#N_2 \colon= (N_1\setminus {\rm int}\,B_i)\cup_f (N_2\setminus {\rm int}\,B_2)$
as the connected sum of $N_1$ and $N_2$. The homeomorphism type of $N_1\#N_2$ does not
depend on the choice of $f$. 
\end{definition}

\begin{definition} An orientable $3$-manifold $N$ is called {\it prime} if it cannot be decomposed
as a nontrivial connected sum of two manifolds, that is, if $N = N_1\#N_2$, then $N_1 = S^3$
or $N_2 = S^3$. A $3$-manifold $N$ is called {\it irreducible} if every its $2$-submanifold homeomorphic to 
a $2$-sphere bounds a subset in $N$ homeomorphic to a closed $3$-ball.
\end{definition}

An orientable, irreducible $3$-manifold is evidently prime. It can be proved that $S^2\times S^1$ is prime, but not irreducible. 
Conversely, if $N$ is a closed orientable prime $3$-manifold, then either
$N$ is irreducible or $N = S^2\times S^1$. 
It is a consequence of the generalized Schoenflies Theorem 
that $S^3$ is irreducible. If $p\colon\widetilde{M}\to M$ is a covering space and $\widetilde{M}$ is irreducible, then $M$ is also irreducible (see \cite[Proposition 1.6]{H07}). In particular, $\mathbb RP^3$ is an irreducible manifold.

The following theorem is due to Kneser (1929), 
Haken (1961), 
and Milnor (1962). 
\begin{theorem}[Prime Decomposition Theorem] {\rm (see \cite[Theorem 1.2.1]{AFW15})} Let $N$ be a compact, oriented $3$-manifold 
with no spherical boundary components. There exist oriented prime $3$-manifolds 
$N_1$, \dots, $N_m$ such that $N=N_1\#\dots\#N_m$. Moreover, if
$N=N_1\#\dots\#N_m$, and $N=N_1'\#\dots\#N_n'$, 
with oriented prime $3$-manifolds $N_i$ and $N_i'$, then $m=n$ and (possibly after reordering)
there exist orientation-preserving homeomorphisms $N_i\to N_i'$.
\end{theorem}
\subsection{The case of manifolds defined by vector-colourings}

Denote by $X^{\#k}$ the connected sum $\underbrace{X\#\dots\#X}_k$.
\begin{proposition}\label{CSprop}
Let $P$ be a connected sum of two simple $3$-polytopes $P_1$ and $P_2$ along vertices, and $M(P,\Lambda)$
be an orientable manifold defined by a vector-colouring $\Lambda$ of rank $r$. Then there is a homeomorphism
$$
M(P,\Lambda)\simeq M(P_1,\Lambda_1)^{\#2^{r-r_1}}\#M(P_2,\Lambda_2)^{\#2^{r-r_2}}\#(S^2\times S^1)^{\#\left[2^{r-3}-2^{r-r_1}-2^{r-r_2}+1\right]},
$$
where $\Lambda_1$ and $\Lambda_2$ are induced vector-colourings of ranks $r_1$ and $r_2$ respectively.
\end{proposition}
\begin{remark}
In the non-orientable case the statement of Proposition \ref{CSprop} is not valid. The reason is that the vectors $\Lambda_i$,
$\Lambda_j$ and $\Lambda_k$ corresponding to the facets $F_i$, $F_j$ and $F_k$ of a $3$-belt may be linearly dependent.
Then $\Lambda$ restricted to $P_1$ and $P_2$ does not satisfy condition (*).
For example, $\Delta^2\times I=\Delta^3\#\Delta^3$. Define $\Lambda$ to be $(1,0,0)$ on bases of the $3$-prism, 
and $(0,1,0)$, $(0,0,1)$, $(0,1,1)$ on side facets. 
\end{remark}
\begin{proof}[Proof of Proposition \ref{CSprop}]
Consider the $3$-belt $\mathcal{B}=(F_i,F_j,F_k)$ of $P$ 
corresponding to the connected sum $P=P_1\#P_2$. Since the facets of the belt are pairwise adjacent, the vectors $\Lambda_i$, $\Lambda_j$ and $\Lambda_k$ are pairwise different. If $\Lambda_i+\Lambda_j+\Lambda_k=0$, 
then $\Lambda_k=\Lambda_i+\Lambda_j$, and $M(P,\Lambda)$ is not orientable, which is a contradiction.
Thus, $\Lambda_i$, $\Lambda_j$, and $\Lambda_k$ are linearly independent. In particular, the
induced mappings $\Lambda_1$ and $\Lambda_2$ for $P_1$ and $P_2$ satisfy condition (*) and are
vector-colourings. 

Let $\pi_i$ be the linear hull of vectors $\Lambda_j\in V=\mathbb Z_2^r$ corresponding to facets of the polytope
$P_i$. Then $\pi_1+\pi_2=V$ and $r_i=\dim \pi_i$.
The $3$-dimensional subspace 
$\pi_{\mathcal{B}}=\langle\Lambda_i,\Lambda_j,\Lambda_k\rangle$ corresponding to $\mathcal{B}$ lies in $\pi_1\cap\pi_2$.

In $M(P,\Lambda)$ the belt $\mathcal{B}$ corresponds to $2^r$ copies of the triangle bounded in $P$ by the middle line of  $\mathcal{B}$. These triangles form $2^{r-3}$ disjoint $2$-spheres. Each sphere is glued of $8$ triangles corresponding to  polytopes $P\times a$ with $a$ lying in the same coset $x+\pi_\mathcal{B}\in V/\pi_{\mathcal{B}}$. 
For the polytopes $P_1$ and $P_2$ each vertex used for the connected sum corresponds to a 
tetrahedron -- the convex hull of this vertex and midpoints
of the three edges incident to the vertex. 
In $M(P_i,\Lambda_i)$ this corresponds to a disjoint set of $2^{r_i-3}$
balls corresponding to polytopes $P_i\times a$ with $a$ lying in the same coset in $\pi_i/\pi_{\mathcal{B}}$. 
Each ball is glued of $8$ tetrahedra.  
In $M(P,\Lambda)$ we have $2^{r-r_i}$ copies of the manifold $N_i$ obtained from $M(P_i,\Lambda_i)$ 
by the deletion of these balls. Each copy is 
glued of parts of  polytopes $P_i\times a$ with $a$ lying in the same coset in $V/\pi_i$ and is bounded by $2^{r_i-3}$ 
spheres.
The boundary spheres of the same copy of $N_i$ 
corresponding to the same coset in 
$\pi_i/(\pi_1\cap\pi_2)$ are also boundary components of the same copy of  
$N_{2-i}$. The spheres corresponding to different cosets in $\pi_i/(\pi_1\cap\pi_2)$
are boundary components of different copies of $N_{2-i}$. Indeed, if the copy of $N_i$ 
is glued of parts of $P_i\times a$ with $a\in x+\pi_i$ and the spheres are glued of triangles lying in the polytopes
$P_i\times b$, where $b\in x+(y+\pi_{\mathcal{B}})$ for the first sphere and $x+(z+ \pi_{\mathcal{B}})$ for the second, and
$y-z\notin \pi_1\cap\pi_2$, then $(x+y)-(x+z)=y-z\not\in\pi_{2-i}$. For otherwise, $y-z\in \pi_1\cap \pi_2$.
If we choose one copy of $N_1$ and one sphere for each coset $\pi_1/(\pi_1\cap\pi_2)$, and one copy of $N_2$ and one sphere for each coset $\pi_2/(\pi_1\cap\pi_2)$ including a unique sphere chosen for $N_1$ and corresponding to $N_2$, then 
these spheres correspond to a connected sum $M(P_1,\Lambda_1)^{\#2^{r-r_1}}\#M(P_2,\Lambda_2)^{\#2^{r-r_2}}$. 
There are $2^{r-r_1}+2^{r-r_2}-1$ such spheres. Each other 
sphere corresponds to an addition of a handle, which is equivalent to a connected sum with $S^2\times S^1$. 
There are $2^r-2^{r-r_1}-2^{r-r_2}+1$
such spheres. This finishes the proof.
\end{proof}

\begin{corollary}\label{RZPcor}
Let $P$ be a connected sum of two simple $3$-polytopes $P_1$ and $P_2$ along vertices. Then there is a homeomorphism
$$
\mathbb R\mathcal{Z}_P\simeq\mathbb{R} \mathcal{Z}_{P_1}^{\#2^{m-m_1}}\#\mathbb{R} \mathcal{Z}_{P_2}^{\#2^{m-m_2}}\#(S^2\times S^1)^{\#\left[(2^{m-m_1}-1)\cdot(2^{m-m_2}-1)\right]},
$$
where $m_i$ is the number of facets of the polytope $P_i$, $i=1,2$.
\end{corollary}
\begin{corollary}\label{sccor}
Let $P$ be a connected sum of two simple $3$-polytopes $P_1$ and $P_2$ along vertices. 
And let $M(P,\Lambda)$ be an orientable small cover. Then there is a homeomorphism
$$
M(P,\Lambda)\simeq M(P_1,\Lambda_1)\#M(P_2,\Lambda_2),
$$
where $\Lambda_1$ and $\Lambda_2$ are characteristic functions for $P_1$ and $P_2$ restricted from $\Lambda$.
\end{corollary}
\begin{example}\label{Dex}
For $P=\Delta^3$ up to a change of basis in the image there are two possible vector-colourings: one of rank $4$ and one of rank $3$. The first of them corresponds to the real moment-angle manifold $\mathbb{R}\mathcal{Z}_{\Delta^3}=S^3$, and the  
second -- to a unique small cover $\mathbb RP^3$. Both of them are irreducible.
\end{example}

\begin{example}\label{DxIex}
Let us classify orientable manifolds $M(P,\Lambda)$ for $P=\Delta^2\times I$. It is necessary that vectors 
$\Lambda_i$, $\Lambda_j$ and $\Lambda_k$ corresponding to side facets of the prism, are linearly independent. 
Let $\Lambda_l$ and $\Lambda_r$ correspond to bases. 

If $\Lambda$ has rank $5$, then $M(P,\Lambda)=\mathbb R\mathcal{Z}_{\Delta^2\times I}=S^2\times S^1$ is a prime manifold. 

If $\Lambda$ has rank $4$, then up to a reflection of $\Delta^2\times I$ we can assume that vectors $\Lambda_i$, $\Lambda_j$, $\Lambda_k$, and $\Lambda_l$ form a basis. Then up to a rotation of $\Delta^3\times I$ we have the following cases:
$\Lambda_r=\Lambda_l$, $\Lambda_r=\Lambda_i+\Lambda_j+\Lambda_k$, and $\Lambda_r=\Lambda_i+\Lambda_j+\Lambda_l$.
Example \ref{Dex} and Proposition \ref{CSprop} imply that $M(P,\Lambda)$ in these cases is homeomorphic to 
$S^3\#S^3\#(S^2\times S^1)=S^2\times S^1$, $S^3\#(\mathbb RP^3)^{\#2}\#(S^2\times S^1)^{\#0}=\mathbb RP^3\#\mathbb RP^3$, 
and $S^3\#S^3\#(S^2\times S^1)=S^2\times S^1$ respectively. 

If $\Lambda$ has rank $3$, then $\Lambda_r=\Lambda_l=\Lambda_i+\Lambda_j+\Lambda_k$, and 
$M(P,\Lambda)=\mathbb RP^3\#\mathbb RP^3\#(S^2\times S^1)^{\#0}=\mathbb RP^3\#\mathbb RP^3$.
\end{example}


\begin{theorem}\label{3th}
\begin{enumerate}
\item \label{I}Any simple $3$-polytope $P$ can be decomposed into a connected sum of simplices and  flag $3$-polytopes 
along vertices:
$$
P=Q_1\#\dots\#Q_N.
$$
The decomposition is unique in the following sense: if there is another decomposition
$P=Q_1'\#\dots\#Q_{N'}'$, then $N=N'$ and there is a reordering of the polytopes and combinatorial equivalences 
$Q_i\simeq Q_i'$ preserving the correspondence between the polytopes, chosen vertices and facets in these vertices under all the connected sums.

\item The orientable manifold $M(\Delta^3,\Lambda)$ is either $S^3$, or $\mathbb RP^3$. It is irreducible. The orientable manifold 
$M(\Delta^2\times I,\Lambda)$ is either $S^2\times S^1$, which is prime, or $\mathbb{R}P^3\#\mathbb{R}P^3$. 
For $P\ne \Delta^3, \Delta^2\times I$ and an orientable manifold $M(P,\Lambda)$ the decompositions from item (\ref{I}) 
and Proposition (\ref{CSprop}) give  the prime decomposition into a connected sum of copies of $S^2\times S^1$ and copies of manifolds $M(Q_i,\Lambda_i)$ of flag $3$-polytopes $Q_i$, where $\Lambda_i$ is the induced vector-colouring. The latter manifolds are aspherical; therefore, they are irreducible.
\end{enumerate}
\end{theorem}
\begin{proof}
The nested set of curves $\widehat{\gamma}(\mathcal{B})$ corresponding to the family of all 
$3$-belts gives a canonical decomposition of $P$ into a connected sum of polytopes $P_i$ 
along triangles and polytopes $Q_i$ along vertices. 
By Corollary \ref{cor:3bQ} the polytopes $Q_i$ have no $3$-belts. Hence each of them is either a flag polytope, or the simplex.

Let us prove the uniqueness of the decomposition. Let $P$ be represented as a connected sum along vertices of polytopes $Q_i'$ without $3$-belts.  Each operation of a connected sum along vertices produces
on the resulting polytope a $3$-belt. Let us choose on this belt a piecewise linear curve satisfying condition
(1) of Definition \ref{def:nsc}. After all the operations we obtain a nested family of curves.  The corresponding belts on $P$ form a nested set. If there is a $3$-belt of $P$ not
lying in this set, then by Corollary \ref{cor:3bQ} it corresponds to a $3$-belt on some polytope $Q_i'$, which is a contradiction. Thus, all the belts of $P$ lie in this set. By Lemma \ref{lem:iso} the nested family of curves corresponding to the polytopes $Q_i'$ is isotopic to the nested family of 
curves corresponding to $Q_i$. Since the polytopes in each family are combinatorially obtained by cutting the polytope $P$ along the curves from the families and shrinking the curves to points, we see that there is a bijection between the sets of 
polytopes $\{Q_i\}$ and $\{Q_i'\}$
preserving combinatorial equivalence. This finishes the proof of the first part of the theorem.

The case of manifolds $M(\Delta,\Lambda)$ and $M(\Delta^2\times I,\Lambda)$ is covered by Examples \ref{Dex} and \ref{DxIex}. Consider the decomposition given by item (\ref{I}) 
and Proposition (\ref{CSprop}). Each manifold in the decomposition is either $M(\Delta,\Lambda)$, which is irreducible, or $S^2\times S^1$, which is prime, or $M(Q_i,\Lambda_i)$ for a flag polytope $Q_i$. For a flag polytope $Q_i$ the manifold
$\mathbb R\mathcal{Z}_{Q_i}$ is aspherical. This follows from \cite[Theorem 2.2.5]{DJS98} (see also \cite[Proposition 1.2.3]{D08}). 
It is known that any aspherical compact closed  orientable $3$-manifold is irreducible. For example,
\cite[Lemma 3.2]{L09} implies that if a  compact closed aspherical $3$-manifold  $M$ is decomposed as a connected sum
$M_1\#M_2$, then one of the summands is homotopy equivalent to $S^3$. Then the Poincare conjecture proved 
by G.~Perelman implies that this summand is homeomorphic to $S^3$.  Since each manifold $M(Q_i,\Lambda_i)$ is covered by 
$\mathbb R\mathcal{Z}_{Q_i}$, it is also aspherical and irreducible.  
\end{proof}

\begin{corollary}
The orientable manifold $M(P,\Lambda)$ over a simple $3$-polytope $P$ is prime if and only if either $P=\Delta^3$, or $P$ is flag,
or $P=\Delta^2\times I$ and $\Lambda$ either has rank $5$, or it has rank $4$ and each vector of the base does not lie in the subspace spanned by vectors of side facets.    
\end{corollary}

\section{$JSJ$-decomposition and geometrization}\label{Sec:main}
\subsection{General information on a $JSJ$-decomposition}\label{JSJgen}
A {\it Seifert fibered manifold} is a compact $3$-manifold $N$ together with a
decomposition of $N$ into disjoint simple closed curves (called Seifert fibers) such
that each Seifert fiber has a tubular neighbourhood that forms a standard fibered torus.
The standard fibered torus corresponding to a pair of coprime integers $(a;b)$ with
$a>0$ is the surface bundle of the automorphism of a disk given by rotation by an
angle of $ \frac{2\pi b}{a}$, equipped with the natural fibering by circles. If $a=1$, then the
middle Seifert fiber is called regular, otherwise it is called singular.


A $3$-manifold $N$ is {\it atoroidal}, if any map
$T\to N$ from a torus to $N$ which induces a monomorphism $\pi_1(T)\to \pi_1(N)$ can be
homotoped into the boundary of $N$. There exist orientable,
irreducible $3$-manifolds that cannot be cut into atoroidal components in a unique way
(e.g., the $3$-torus). Nonetheless,  any compact, orientable, irreducible
$3$-manifold with empty or toroidal boundary admits a canonical decomposition along
tori into atoroidal and Seifert fibered manifolds.

\begin{definition}
A $2$-submanifold $\Sigma$ of a $3$-dimensional manifold $N$ is called {\it incompressible}, if each its 
connected component $\Sigma_i$ is not homeomorphic to $S^2$ and $D^2$ and each induced mapping 
$\pi_1(\Sigma_i)\to \pi_1(N)$ is injective.
\end{definition}

The following theorem was first announced by Waldhausen (1969) 
and proved independently by Jaco--Shalen (1979) 
and Johannson (1975, 1979).

\begin{theorem}[JSJ-Decomposition Theorem] {\rm (see \cite[Theorem 1.6.1]{AFW15})} 
Let $N$ be a compact, orientable, irreducible
$3$-manifold with empty or toroidal boundary. There exists a (possibly empty)
collection of disjointly embedded incompressible tori $T_1$, \dots, $T_m$ such that each component
of $N$ cut along $T_1$, \dots, $T_m$ is atoroidal or Seifert fibered. Any such collection
of tori with a minimal number of components is unique up to isotopy.
\end{theorem}

\begin{notation}
Denote by $K^2$ the Klein bottle, and by $K\widetilde{\times} I$ the unique orientable total space of (a necessarily twisted)
$I$-bundle over $K^2$.
\end{notation}

The following theorem 
says in particular that "sufficiently complicated"
Seifert fibered manifolds carry a unique Seifert fibered structure.
\begin{theorem}\label{sun-th}{\rm (see \cite[Theorem 1.5.2]{AFW15})} Let $N$ be an orientable $3$-manifold that admits a Seifert fibered
structure. If $N$ has nontrivial boundary, then the Seifert fibered structure of $N$ is
unique up to isotopy if and only if $N$ is not homeomorphic to one of the following: 
$S^1\times D^2$, $T^2\times I$, and $K^2\widetilde{\times}I$.
\end{theorem}

For an additively written abelian group $G$ we denote by $PG$ the set of equivalence
classes of the equivalence relation $g\sim-g$ on $G\setminus\{0\}$, and for $g\in G\setminus\{0\}$ 
we denote by $[g]$ its equivalence class in $PG$. Given a Seifert fibered manifold $N$ with a choice
of a Seifert fibered structure and a boundary torus $T$, the regular fiber determines
an element $f(N,T)\in PH_1(T)$ in a natural way. It follows from Theorem \ref{sun-th}
that $f(N,T)$ is independent of the choice of the Seifert fibered structure, unless $N$ is one
of $S^1\times D^2$, $T^2\times I$, $K^2\widetilde{\times} I$.
For later use we record the following.

\begin{lemma} {\rm (see \cite[Lemma 1.5.3]{AFW15})} Let $N_1$ and $N_2$ be two Seifert fibered manifolds (where we allow
$N_1=N_2$) and let $N_1\cup_{T_1=T_2}N_2$ be the result of gluing $N_1$ and $N_2$ along boundary
tori $T_1$ and $T_2$. Then the following are equivalent:
\begin{enumerate}
\item there exists a Seifert fibered structure on $N_1\cup_{T_1=T_2}N_2$ that restricts (up to
isotopy) to the Seifert fibered structures on $N_1$ and $N_2$;
\item $f(N_1,T_1)=f(N_2,T_2)\in PH_1(T_1)=PH_1(T_2)$.
\end{enumerate}
\end{lemma}

There is a criterion for showing that a collection of tori in $N$
are in fact the $JSJ$-tori. 

\begin{proposition}\label{JSJcrit} {\rm (see \cite[Proposition 1.6.2]{AFW15})}
Let $N$ be a compact, orientable, irreducible $3$-manifold with
empty or toroidal boundary. Let $T_1$, $\dots$, $T_m$ be disjointly embedded incompressible
tori in the $3$-manifold $N$. The $T_i$ are the $JSJ$-tori of $N$ if and only if
\begin{enumerate}
\item the components $M_1$, $\dots$, $M_n$ of $N$ cut along $T_1\cup\dots\cup T_m$ 
are Seifert fibered or atoroidal;
\item for any choice of Seifert fibered structures for all the Seifert fibered components
among $M_1$, \dots $M_n$ the following two conditions hold:
\begin{enumerate}
\item  if a torus $T_i$ cobounds two Seifert fibered components $M_r$ and $M_s$ with
$r\ne s$, then $f(M_r,T_i)\ne f(M_s,T_i)\in PH_1(T_i)$ and
\item if a torus $T_i$ abuts two boundary tori $T_i'$ and $T_i''$ of the same Seifert
fibered component $M_r$, then $f(M_r,T_i')\ne f(M_r,T_i'')$,
\end{enumerate}
and
\item if one of the $M_i$ equals $T^2\times I$, then $m=n=1$, i.e., $N$ is a torus bundle.
\end{enumerate}
\end{proposition}

To distinguish between Seifert fibered and atoroidal components we will use the following result.
\begin{theorem}(see \cite[Theorem 2.10]{B02})\label{hypnoSei}
A complete hyperbolic manifold of finite volume can not be a Seifert fibered manifold.
\end{theorem}

\subsection{A connection between the $JSJ$-decomposition and a geometrization}\label{JSJgeom}

A connection between the $JSJ$-decomposition and the Geometric Decomposition Theorem \ref{GDT} 
is described in the following way. 
We will refer to the union of tori and Klein bottles from Theorem
\ref{GDT} as the {\it  geometric decomposition surface}.

\begin{proposition} {\rm (see \cite[Proposition 1.9.2]{AFW15})}. \label{GDP}
Let $N$ be a compact closed orientable, irreducible $3$-manifold.
\begin{enumerate}
\item If $N$ is a $Sol$-manifold, then $N$ is geometric, i.e., the geometric decomposition
surface is empty. On the other hand $N$ has one $JSJ$-torus, namely, if
$N$ is a torus bundle, then the fiber is the $JSJ$-torus, and if $N$ is a twisted
double of $K^2\widetilde{\times}I$, then the JSJ-torus is given by the boundary of $K^2\widetilde{\times}I$.

\item Suppose that $N$ is not a $Sol$-manifold, and denote by $T_1$, \dots, $T_m$ the $JSJ$-tori
of $N$. We assume that they are ordered such that the tori $T_1$, \dots $T_n$ do not
bound copies of $K^2\widetilde{\times}I$ and that for $i=n+1,\dots, m$, each $T_i$ cobounds a
copy of $K^2\widetilde{\times}I$. Then the geometric decomposition surface of $N$ is given by
$$
T_1\cup\dots\cup T_n\cup K_{n+1}\cup\dots\cup K_m.
$$
Conversely, if $T_1\cup\dots\cup T_n\cup K_{n+1}\cup\dots\cup K_m$
is the geometric decomposition surface such that $T_1$, \dots, $T_n$ are tori and $K_{n+1}$, \dots, $K_m$ are Klein bottles, 
then the $JSJ$-tori are given by 
$$
T_1\cup \dots\cup T_n\cup\partial \nu(K_{n+1})\cup\dots\cup\partial \nu(K_{n+1}),
$$
where $\nu(K_i)$ is a tubular neighbourhood of $K_i$.
\end{enumerate}
\end{proposition}
\begin{remark}
To use item (1) of Proposition \ref{GDP} we will not give details on $Sol$-manifolds. The only facts we will
need is that for a torus bundle the complement to the $JSJ$-torus is homeomorphic to $T^2\times I$, and for the  twisted
double of $K^2\widetilde{\times}I$ the complement of the $JSJ$-torus consists of two copies if $K^2\widetilde{\times}I$
(see \cite[Section 1.5]{AFW15}).
\end{remark}
All the previous results and the appropriate references can be found in \cite{AFW15}.

\subsection{Application to manifolds defined by vector-colourings}\label{MPJSJ}
Now, using general results described above we will find the $JSJ$-decomposition and the minimal geometric decomposition
of any orientable manifold $M(P,\Lambda)$ defined by a vector-colouring of a simple $3$-polytope. 
It turns out, that it corresponds to a canonical decomposition of a flag $3$-polytope along $4$-belts
into $k$-prisms, which correspond to Seifert fibered pieces, and almost Pogorelov polytopes, which produce hyperbolic parts.
The $JSJ$-tori, which are defined uniquely up to an isotopy, correspond to these canonical belts, and also to "free quadrangles"
of almost Pogorelov polytopes. These quadrangles are not surrounded by the canonical $4$-belts.  An 
explicit geometrization of the pieces may be given by a construction invented by 
A.D.~Mednykh and A.Yu.~Vesnin in a series of papers \cite{MV86, V87, M90, VM99, V17}. 
\begin{construction}\label{MVconstr}
Let $P$ be a right-angled polytope of finite volume in $\mathbb R^3$, $S^3$, $\mathbb{H}^3$, $S^2\times \mathbb R$ or $\mathbb{H}^2\times \mathbb R$, were for a product of spaces we mean that $P$ is a product of right-angled polytopes of finite volume in these spaces. For $\mathbb R^3$ such a polytope is combinatorially a cube, for $S^3$ -- a simplex, for 
$\mathbb{H}^3$ -- an almost Pogorelov polytope without adjacent quadrangles with all the quadrangles shrinked to points, for $S^2\times \mathbb R$ the polytope is a product of a right-angled triangle and a segment, and for $\mathbb{H}^2\times \mathbb R$ this is a product of a right-angled polygon of finite area and a segment. The polygon may have vertices on the absolute, while all its proper vertices have right
angles between the corresponding edges.

The polytope $P$ corresponds to a right-angled coxeter group 
$$
\langle \rho_1,\dots,\rho_m\rangle/(\rho_1^2,\dots,\rho_m^2, \rho_i\rho_j=\rho_j\rho_i,
\text{ if }F_i\cap F_j\text{ is a proper edge}).
$$
This group is isomorphic to a subgroup $G(P)$ of isometries of the ambient space generated by reflections in facets of $P$,
where $\rho_i$ corresponds to a reflection in $F_i$. The group $G(P)$ acts on the ambient space discretely, and the polytope is a fundamental domain, that is the interiors of the polytopes $\{gP\}_{g\in G(P)}$ do not intersect. Also $P$ is an orbit space, and for any point of $P$ its stabilizer is generated by reflections in facets containing this point (see \cite[Theorem 1.2 in Chapter 5]{VS88}).

Given a mapping  $\Lambda\colon\{F_1,\dots,F_m\}\to\mathbb Z_2^r\setminus\{0\}$ such that
for any proper edge $F_i\cap F_j$ and any proper vertex  $F_i\cap F_j\cap F_k$ the images of the 
corresponding facets are linearly independent and ${\rm Im\,}\varphi_{\Lambda}$ spans $\mathbb Z_2^r$
define a mapping $\varphi_{\Lambda}\colon G(P)\to \mathbb Z_2^r$
by the rule $\varphi_{\Lambda}(\rho_i)=\Lambda(F_i)$.  

The subgroup ${\rm Ker\,} \varphi_{\Lambda}$ acts freely on the ambient space (excluding points at infinity), 
and the quotient space $N(P,\Lambda)$ 
is a finite volume manifold glued of $2^r$ copies of the polytope. This manifold automatically has a geometric structure modelled on the ambient space.

If $P$ is a compact polytope, then the mapping $\Lambda(F_i)=e_i\in\mathbb Z_2^m$ produces
$N(P,\Lambda)=\mathbb R\mathcal{Z}_P$, and a vector-colouring $\Lambda$ produces $N(P,\Lambda)=M(P,\Lambda)$.
The inverse homeomorphism is given by the mapping $(p,t)\to \varphi^{-1}_{\Lambda}(t)\cdot p$. 
\end{construction}
Remind that the definitions of a nested family of belts, a nested family of curves, and cutting a polytope along a belt and a nested family of belts were given in Definitions \ref{NSB}, \ref{def:nsc}, and Construction \ref{con:bcut} together with Remark \ref{famcut} respectively. Using these notions we 
can formulate our main result.
\begin{theorem}\label{M-th}
Let $P$ be a simple flag $3$-polytope different from $I^3$. Then
\begin{enumerate}
\item There is a unique nested family of $4$-belts such that
cutting  $P$ along all these belts gives 
\begin{enumerate} 
\item almost Pogorelov polytopes without adjacent quadrangles, and 
\item $k$-prisms, $k\geqslant 5$, where any two adjacent prisms ``are twisted'', that is if two prisms have quadrangles corresponding to the same belt in $P$,
then their bases correspond to different facets of this belt. 
\end{enumerate}
We call the belts from this family canonical. We call a quadrangle of the obtained almost Pogorelov polytope free, if 
it does not correspond to canonical belts. 
\item For any  orientable manifold $M(P,\Lambda)$ the family of the canonical belts gives  
the $JSJ$-decomposition and the minimal geometric decomposition in the following way:
\begin{enumerate}
\item there is a geometric realization of $P$ in $\mathbb R^3$ such that plane sections of the canonical belts are disjoint quadrangles bounded by curves forming a nested family; 
\item the deletion of these quadrangles and all the quadrangles of $P$ corresponding to free quadrangles 
decomposes $P$ into a disjoint union of parts $Q_i$, $i=1,2,\dots$. Each $Q_i$ is homeomorphic to a right-angled polytope of 
finite volume either in $\mathbb {H}^2\times \mathbb R$ (if $Q_i$ corresponds to a prism), or  $\mathbb{H}^3$ 
(if $Q_i$ corresponds to an almost Pogorelov polytope), and has an induced vector-colouring
$\Lambda_{Q_i}$;
\item  the $JSJ$-tori of $M(P,\Lambda)$ are the connected components of the preimages of 
\begin{enumerate}
\item the quadrangles of  the canonical belts;
\item the quadrangles $F_j$ of $P$ corresponding to free quadrangles such that the colouring vector $\Lambda_j$ does not lie in the subspace spanned by the vectors of adjacent facets;
\item\label{III}  slight shifts inside $P$ of  the quadrangles $F_j$ of $P$ corresponding to free quadrangles such that $\Lambda_j$ lies
in the subspace spanned by the vectors of adjacent facets.
\end{enumerate}
Thus, the set of $JSJ$ tori is divided into three types. Each torus of type $(iii)$ is a boundary of a tubular neighbourhood of the Klein bottle, which is a connected component of the preimage of $F_j$;
\item the tori of types $(i)$ and $(ii)$, and the Klein bottles corresponding to the tori of type $(iii)$ give the minimal geometric decomposition. Each piece of the complement to these surfaces in $M(P,\Lambda)$ is homeomorphic to a manifold $N(Q_i,\Lambda_i)$ with the geometric structure  of finite volume given by Construction \ref{MVconstr}.
\end{enumerate}
\end{enumerate}
\end{theorem}
\begin{remark}
There are two extreme cases, when the set of canonical belts is empty.
The first case is when $P$ is a $k$-prism, $k\geqslant 4$. Then the orientable manifold $M(P,\Lambda)$ is a compact 
closed Seifert fibered manifold, which already has a geometric structure. 
For $k=4$ this is $\mathbb R^3$, and for $k\geqslant 5$
this is $\mathbb H^2\times\mathbb R$.  The set of $JSJ$-tori in item (2) in this case is empty.
The second case is when $P$ is an almost Pogorelov polytope without adjacent quadrangles. If $P$ is a Pogorelov polytope, 
then the orientable manifold $M(P,\Lambda)$ is a compact closed hyperbolic manifold, and the set of $JSJ$-tori in item (2) is
also empty. If $P$ has quadrangles, then in item (2) each quadrangle corresponds to $JSJ$-tori, and $M(P,\Lambda)$
is divided into homeomorphic pieces with hyperbolic structure of finite volume.
\end{remark}
\begin{remark}
Item (1) of this theorem is a partial case of a more general decomposition of so-called Coxeter orbifolds in \cite{S09}. One should take a right-angled Coxeter orbifold.
\end{remark}
\begin{remark}
Item (1) is equivalent to the fact that $P$ can be uniquely decomposed as a connected sum 
of $k$-prisms, $k\geqslant5$, and almost Pogorelov polytopes without adjacent quadrangles in such a way that
adjacent prisms are twisted. Indeed, cutting along a $4$-belt is an inverse operation to a connected sum along quadrangles. Therefore, item (1) automatically gives the described decomposition. On the other hand, if we take a collection of  $k$-prisms, $k\geqslant 5$, and almost Pogorelov polytopes without adjacent quadrangles and start taking connected sums of these polytopes along quadrangles following the twisted prisms rule, 
then at each step the arising $4$-belt can not contain quadrangles. In particular, the arising $4$-belts
form a nested family and each quadrangle
of the new polytope corresponds to a quadrangle of a unique summand. This is clear for the case of two polytopes. Assume that 
we have proved this fact for each polytope obtained as a connected sum of less than $k$ polytopes from the collection. 
Consider the case of $k$ polytopes. The last connected sum involves two polytopes $Q_1$ and $Q_2$ 
each being either a polytope from the collection or a connected sum of less than $k$ such polytopes. 
By induction the chosen quadrangles on $Q_1$ and $Q_2$ correspond to uniquely defined polytopes $P_1$ and $P_2$
from the collection. If the $4$-belt arising in the last connected sum contains a quadrangle,
then $P_1$ and $P_2$ are prisms. Since they should be twisted, we obtain a contradiction.
\end{remark}

We will prove theorem Theorem \ref{M-th} is several steps.

\subsubsection{The decomposition of a flag polytope.}\label{flagdec}
We start with a lemma.
\begin{lemma}
Let a simple $3$-polytope $P$ be a connected sum of polytopes $P_1$ and $P_2$ along 
$k$-gons surrounded by belts, $k\geqslant 4$. Then $P$ is flag if and only if $P_1$ and $P_2$ are flag.
\end{lemma}
\begin{proof}
This is a direct consequence of Lemma \ref{Piblemma} and Corollary \ref{Pibcor}.
\end{proof}

In the case of $4$-belts the situation is more complicated then for $3$-belts. 
It is possible that the middle lines of $4$-belts intersect transversally. For example, for any $k$-prism, $k\geqslant 4$, there are such $4$-belts (in particular, the belts around adjacent quadrangles). As we will see below, this is almost the general 
``bad'' case.

{\bf Algorithm  for the decomposition.} 
First we will find algorithmically the decomposition (1) of Theorem \ref{M-th}, and then prove that it is unique.

{\bf Step 1.}  We cut off all sequences of adjacent quadrangles. 
\begin{lemma}
If a vertex of a flag $3$-polytope $P$ belongs to three quadrangles, then $P=I^3$.
\end{lemma} 
\begin{proof}
Indeed, these three quadrangles are surrounded by a sequence of $3$ pairwise
adjacent facets. Since $P$ has no $3$-belts, they should have a common vertex. Hence, $P=I^3$.
\end{proof}

Assume that $P\ne I^3$. Then each its vertex belongs to at most two quadrangles. Let there be two adjacent 
quadrangles $F_i$ and $F_j$.  If a quadrangle $F_q$ is adjacent to one of them, then their edge of intersection should be 
opposite to $F_i\cap F_j$ in this facet. Thus, moving in each direction we obtain a sequence of quadrangles with ``parallel''
edges of intersection. This sequence either becomes a ring, or stops at some moment. In the first case $P$ is 
a $k$-prism for $k\geqslant 4$. In the second case let us cut the polytope along the $4$-belt 
surrounding a sequence of quadrangles to obtain a $k$-prism and some new polytope $P_1$. 
For $P_1$ apply the same procedure. Since the belts we cut along do not pass quadrangles,
we see that all the arising belts form a nested family. In the end we obtain a polytope $Q$ without adjacent quadrangles 
or a $k$-prism, $k\geqslant 5$. Each step decreased the number of facets at least by one. 
If $Q$ has only trivial $4$-belts or $Q$ is a $k$-prism we move to Step 3. Otherwise, we move to Step 2. 

{\bf Step 2.} We cut a flag $3$-polytope without adjacent quadrangles along a nontrivial $4$-belt passing no quadrangles. 
\begin{lemma}
Let $Q$ be a flag $3$-polytope without adjacent quadrangles. Then either $Q$ has only trivial $4$-belts, or
it has a $4$-belt passing no quadrangles.
\end{lemma}
\begin{proof}
Let $\mathcal{B}=(F_i,F_j,F_k,F_l)$ be a nontrivial $4$-belt. 
Then it contains at most two quadrangles. Let $F_i$ be one of them. In the closure of
each component of the complement $\partial Q\setminus|\mathcal{B}|$ the quadrangle $F_i$ is adjacent to some 
facets $F_p$ and $F_q$ such that
$(F_j,F_p,F_l,F_q)$ is a $4$-belt (each facet of a flag $3$-polytope is surrounded by a belt). If $F_p\cap F_k\ne\varnothing$,
then $F_p\cap F_k\cap F_j$ and $F_p\cap F_k\cap F_l$ are vertices, and $\mathcal{B}$ is a belt around $F_p$. 
A contradiction. Similarly, $F_q\cap F_k=\varnothing$. Thus, $\mathcal{B}_1=(F_p,F_j,F_k,F_l)$
is a $4$-belt and $F_p$, $F_j$ and $F_l$ are not quadrangles. The belt $\mathcal{B}_1$ is nontrivial, for otherwise 
$\mathcal{B}$ surrounds two adjacent quadrangles. Similarly, 
$\mathcal{B}_2=(F_q,F_j,F_k,F_l)$ is a nontrivial $4$-belt. If $F_k$ is a quadrangle, then we can apply to it a similar argument 
to obtain nontrivial belts containing no quadrangles. 
\end{proof}
Thus, if $Q$ is not an almost Pogorelov polytope and not a $k$-prism, then we can cut it along a belt containing no quadrangles
to obtain two flag polytopes $Q_1$ and $Q_2$ with less number of facets. If each polytope $Q_i$ is an almost Pogorelov polytope or a $k$-prism, then we move to Step 3. Otherwise, for each polytope $Q_i$ we move to Step 1.

{\bf Step 3.} Now we have cut a polytope $P$ along a nested set of $4$-belts into almost Pogorelov polytopes without adjacent 
quadrangles and $k$-prisms,
$k\geqslant 5$ (the case $k=4$ is impossible, since it corresponds to cutting along a trivial $4$-belt). This corresponds to a representation of $P$ as a connected sum of the corresponding polytopes along quadrangles. Assume that in this decomposition
two prisms are glued along quadrangles. If their bases are glued along edges, then we  obtain a prism again. We delete the corresponding belt from a nested family. Repeating this argument in the end we come to 
the case when any two adjacent prisms are ``twisted''. Here the algorithm stops.

{\bf Uniqueness of the decomposition.} Now let us prove that the decomposition from item (1)
of Theorem \ref{M-th} is unique.  Similarly to the case of $3$-belts we can combinatorially represent the decomposition as follows. 
For the $4$-belts we take the curves $\widehat{\gamma}(\mathcal{B})$ from Construction \ref{con:hgamma}.
Then we cut $\partial P$ along these curves and glue up each curve by a quadrangle to obtain boundaries of the polytopes $R_i$.
Again for a facet $F_a$ of $P$ we denote by $\widehat{F_a}$ the corresponding facet of $R_i$. The facet 
$\widehat{F_a}$ can be considered as a part of $F_a$, and $F_a$ is assembled from such parts.


\begin{lemma}\label{4bRlemma}
Let $\mathcal{F}$ be a nested family of $4$-belts on a flag $3$-polytope $P\ne I^3$ such that cutting along all of these 
belts gives polytopes $R_i$, which are either $k$-prisms, $k\geqslant 5$, or almost Pogorelov polytopes without 
adjacent quadrangles, and any two adjacent prisms ``are twisted''. Then 
\begin{enumerate}
\item any $4$-belt in $\mathcal{F}$ contains no quadrangles;
\item any $4$-belt of $P$ either belongs to $\mathcal{F}$ or is compatible with all the belts in $\mathcal{F}$. 
In the latter case it consists of facets presented simultaneously in a unique polytope $R_i$, where they form a $4$-belt.
\end{enumerate}
\end{lemma}
\begin{proof}
If a belt $\mathcal{B}\in\mathcal{F}$ contains a quadrangle $F_i$, then $\widehat{\gamma}(\mathcal{B})\cap F_i=E$ is a line segment and any nonempty intersection $\widehat{\gamma}(\mathcal{B}')\cap F_i\ne\varnothing$ for 
$\mathcal{B}'\in\mathcal{F}\setminus\{\mathcal{B}\}$ is a line segment "parallel" to $E$.
Thus, if $F_i$ is presented in some polytope $R_j$, then the corresponding facet $\widehat{F_i}$ 
is also a quadrangle. Let $\mathcal{B}$
separate the polytopes $R_p$ and $R_q$. Then at each of these polytopes $F_i$ corresponds to a quadrangle adjacent to the quadrangle arising from $\mathcal{B}$. Then both $R_p$ and $R_q$ are prisms.
Since they are twisted, for one of these polytopes $\widehat{F}_i$ is a base. Hence, it is not a quadrangle. A contradiction.

Let us prove that a $4$-belt on $P$ different from the belts in $\mathcal{F}$ 
can not be presented in different polytopes $R_i$ and $R_j$. Indeed, 
any two polytopes in the decomposition are separated on $P$ by some of the curves $\widehat{\gamma}(\mathcal{B})$,
$\mathcal{B}\in\mathcal{F}$. 
If a facet is presented in both polytopes, then it belongs to $\mathcal{B}$. Hence, the belt should coincide with $\mathcal{B}$. 

Now consider any $4$-belt $\mathcal{B}=(F_i,F_j,F_k,F_l)\notin\mathcal{F}$.

If $\mathcal{B}$ is compatible with any belt in $\mathcal{F}$,
then using Construction \ref{con:hgamma} and Lemma \ref{lem:disj} we can move each vertex of 
the middle line $\gamma(\mathcal{B})$
in the interior of the corresponding edge of $P$ in such a way that the new curve $\widehat{\gamma}(\mathcal{B})$ 
does not intersect the curves $\widehat{\gamma}(\mathcal{B}')$ for all $\mathcal{B}'\in\mathcal{F}$. 
In particular, it lies in the interior of 
a part of $\partial P$ corresponding to some $R_i$.
Then all the facets of $\mathcal{B}$ are presented in $R_i$.  Lemma \ref{intlemma} implies that the facets $F_a$ and $F_b$
presented in $R_i$ are adjacent in $P$ if and only if the facets $\widehat{F_a}$ and $\widehat{F_b}$ are adjacent in $R_i$. Hence,  the corresponding facets of $R_i$ also form a $4$-belt.

Assume that $\mathcal{B}$ is not compatible with some belt $\mathcal{B}'\in\mathcal{F}$ connecting polytopes $R_p$ and $R_q$. Then Lemma \ref{ncb-lemma} implies that for some relabelling of facets 
$F_j$ lies in the closure of a connected component of $\partial P\setminus|\mathcal{B}'|$ corresponding to $R_p$, and $F_l$
 -- to $R_q$. 
Then $F_i$ and $F_k$ should belong to $\mathcal{B}'$. We claim that in this case at least one of the facets 
$F_j$ and $F_l$ does not intersect one of the facets $F_i$ and $F_k$. 
Indeed, consider the polytope $R_p$.  If it is not a prism, then it is an almost Pogorelov polytope without adjacent quadrangles. 
Then there is exactly one quadrangle in $R_p$ adjacent to both $\widehat{F_i}$ and $\widehat{F_k}$, 
namely the quadrangle $F$
corresponding to the belt $\mathcal{B}'$. If there is another quadrangle $F'$ with this property, then 
$(\widehat{F_i},F,\widehat{F_k},F')$ is a $4$-belt. Since $R_p$ has only trivial $4$-belts, this belt should surround a quadrangle
adjacent  to both $F$ and $F'$. A contradiction.  
Therefore, each of the parts of $\partial P$ separated from $R_p$ by curves
$\widehat{\gamma}(\mathcal{B}'')$, where $\mathcal{B}''\in \mathcal{F}\setminus\{\mathcal{B}'\}$ are
$4$-belts presented in $R_p$, intersects at most one of the facets 
$F_i$ and $F_k$. Thus, if $F_j$ is not presented in $R_p$, it can not intersect
both $F_i$ and $F_k$. If it is presented in $R_p$, then $\widehat{F_j}$ is adjacent to both 
$\widehat{F_i}$ and $\widehat{F_k}$ in this polytope. 
The sequence $(F,\widehat{F_i},\widehat{F_j},\widehat{F_k})$
can not be a $4$-belt. For otherwise, this belt surrounds a quadrangle adjacent to $F$. Thus, $\widehat{F_j}\cap F\ne\varnothing$. Then
$F_j$ belongs to $\mathcal{B}'$. A contradiction. So, $F_j$ can not intersect both $F_i$ and $F_k$, if $R_p$ is not a prism.
Similarly, $F_l$ can not intersect both $F_i$ and $F_k$, if $R_q$ is not a prism.

If both $R_p$ and $R_q$ are prisms, then in one of them, say in $R_p$, $\widehat{F_i}$ and $\widehat{F_k}$ are quadrangles, 
and in the other, respectively $R_q$, they are bases.
In this case in $R_p$ there is only one quadrangle  adjacent to  both  $\widehat{F_i}$ and $\widehat{F_k}$, 
the quadrangle $F$ corresponding to $\mathcal{B}'$. Therefore, each of the parts of $\partial P$ separated from $R_p$ by curves
$\widehat{\gamma}(\mathcal{B}'')$, where $\mathcal{B}''\in \mathcal{F}\setminus\{\mathcal{B}'\}$ are
$4$-belts presented in $R_p$, intersects at most one of the facets 
$F_i$ and $F_k$.
Therefore, if the facet $F_j$ is not presented in $R_p$, 
it can not intersect both $F_i$ and $F_k$. If it is presented in $R_p$, then $\widehat{F_j}$ is adjacent to both $\widehat{F_i}$ and 
$\widehat{F_k}$ in this polytope. Then $\widehat{F_j}$ is a base and $F_j$ belongs to 
$\mathcal{B}'$. A contradiction. 
\end{proof}

Now let $\mathcal{F}'$  be another nested family  of $4$-belts such that cutting along all of these belts gives polytopes $R_i'$, which are either $k$-prisms, $k\geqslant 5$, or almost Pogorelov polytopes without adjacent quadrangles, and any two 
adjacent prisms ``are twisted''.  By Lemma \ref{4bRlemma} the union $\mathcal{F}'\cup \mathcal{F}$
is a nested family of belts again. Using Construction \ref{con:hgamma} we can choose a nested family of curves
corresponding to $\mathcal{F}'\cup\mathcal{F}$ as an extension of such a family for $\mathcal{F}$.
For each belt $\mathcal{B}\in \mathcal{F}'\setminus\mathcal{F}$ the curve $\widehat{\gamma}(\mathcal{B})$ 
lies on some polytope $R_i$ and corresponds to a $4$-belt $\widehat{\mathcal{B}}$ on it. Assume that $\widehat{\mathcal{B}}$
is a trivial belt. Then it can not surround a quadrangle corresponding to belts in $\mathcal{F}$. Thus, 
it surrounds some quadrangle $\widehat{F_j}\in R_i$. This quadrangle is adjacent only to facets in $\widehat{\mathcal{B}}$,
and all of them have the form $\widehat{F_k}$ for some facets $F_k$ of $P$. Hence, $F_j$ does not intersect 
the curves $\widehat{\gamma}(\mathcal{B}')$, $\mathcal{B}'\in \mathcal{F}$.  Therefore, $F_j$ is a quadrangle, as well as $\widehat{F_j}$. Then the component of $\partial P\setminus\gamma(\mathcal{B})$ containing $F_j$ can not contain 
curves $\gamma(\mathcal{B}')$ for any $4$-belt $\mathcal{B}'$ of $P$. In particular, when we cut along the belts 
in $\mathcal{F}'$ one of the polytopes $R_j'$ is a $4$-prism, which is a contradiction.
Thus, $\widehat{\mathcal{B}}$ is a nontrivial belt. Hence, the polytope $R_i$ is a prism, and $\widehat{\mathcal{B}}$ contains two
bases and two quadrangles. Consider all the belts $\mathcal{B}\in\mathcal{F}'\setminus\mathcal{F}$ corresponding to
nontrivial belts in $R_i$. If we cut $R_i$ along the corresponding curves $\widehat{\gamma}(\mathcal{B})$, we obtain a set of $k$-prisms
with $k\geqslant 5$  such that any two adjacent prisms are not twisted. Taking such cuts for all the prisms $R_i$
we obtain a set of prisms. If we add all the almost Pogorelov polytopes without adjacent quadrangles $R_j$, we obtain the 
family of polytopes arising from $P$ if we cut it along all the curves $\widehat{\gamma}(\mathcal{B})$, $\widehat{\gamma}(\mathcal{B})\in\mathcal{F}\cup\mathcal{F'}$. Now if do the same starting with $\mathcal{F}'$ and then
considering $\mathcal{F}\setminus\mathcal{F}'$, then we will obtain the same family of polytopes. 
Now consider two adjacent prisms
corresponding to the above polytope $R_i$. In the second way of cutting $P$ they arise from different prisms $R_p'$ and $R_q'$, therefore they should be twisted. A contradiction. Thus,  $\mathcal{F}=\mathcal{F}'$. 

\subsubsection{Incompressible surfaces corresponding to belts and facets.}\label{BFinc}
For a subset $\omega\subset[m]$ the simplicial complex 
$K_{\omega}=\{\sigma\in K\colon \sigma\subset\omega\}$ is called a {\it full (or induced) subcomplex}. 

In $\mathbb{R}\mathcal{Z}_K$ the simplices of the full subcomplex
correspond to the subset of the form $\mathbb R\mathcal{Z}_{K_{\omega}}\times \mathbb Z_2^{m-|\omega|}$. This is a disjoint
union of $2^{m-|\omega|}$ copies of $\mathbb R\mathcal{Z}_{K_{\omega}}$. 

The following result was communicated to the author by Taras Panov.
\begin{lemma}\label{rlemma}\cite[Exercise 4.2.13]{BP15}, \cite[Proposition 2.2]{PV16}
For any subset $\omega\subset[m]$ and any $a\in\mathbb Z_2^{m-|\omega|}$ there is a retraction 
$\mathbb R\mathcal{Z}_K\to\mathbb R\mathcal{Z}_{K_{\omega}}\times a$.
\end{lemma}
\begin{proof}
Indeed, the retraction has the form $(x_1,\dots,x_m)\to (y_1,\dots,y_m)$, where 
$y_i=\begin{cases}a_i,& i\notin \omega;\\x_i,&i\in\omega\end{cases}$. 
This mapping is induced by the projection of the cube $(D^1)^m$ to its face $(D^1)^{|\omega|}\times a$. It maps 
$(D^1,S^0)^{\sigma}=(D^1)^{|\sigma|}\times (S^0)^{m-|\sigma|}\to (D^1)^{|\sigma\cap \omega|}\times a\subset (D^1,S^0)^{\sigma\cap\omega}\subset \mathbb R\mathcal{Z}_{K_\omega}\times a$ and by construction 
is continuous.  Also it is identical on  $\mathbb R\mathcal{Z}_{K_{\omega}}\times a$. 
\end{proof}
\begin{corollary}\label{Zwinc}
For any subset $\omega\subset[m]$ and any $a\in\mathbb Z_2^{m-|\omega|}$ the mapping of fundamental groups 
$\pi_1(\mathbb R\mathcal{Z}_{K_\omega}\times a)\to \pi_1(\mathbb R\mathcal{Z}_K)$ is injective.
\end{corollary}
Any $k$-belt $\mathcal{B}$ of a simple $3$-polytope $P$ corresponds to the set $\omega(\mathcal{B})=\{i\colon F_i\in \mathcal{B}\}$.
The induced subcomplex $K_{\omega(\mathcal{B})}$ in $K=\partial P^*$ is a simple cycle, which is a boundary of a $k$-gon. 
This $k$-gon may be realized as a piecewise linear subspace $P_k(\mathcal{B})$ 
of $P$ consisting of triangles with vertices the barycentre of 
$P$, the barycentre of a facet of the belt, and the barycentre of its edge of intersection with a successive facet of the belt.

\begin{definition}
For a vector-colouring of a simple $3$-polytope $P$ of rank $r$ and a $k$-belt $\mathcal{B}$ define 
 $\pi_{\mathcal{B}}=\langle \Lambda_j\colon F_j\in\mathcal{B}\rangle$ and $r(\mathcal{B})=\dim \pi_{\mathcal{B}}$.
Then there is an {\it induced vector-colouring} 
$\Lambda_{\mathcal{B}}$ of $P_k(\mathcal{B})$ of rank $r(\mathcal{B})$, $2\leqslant r(\mathcal{B})\leqslant r$. 
\end{definition}

\begin{proposition}\label{Binprop}
Let $M(P,\Lambda)$ be a manifold defined by a vector-colouring $\Lambda$ of rank $r$ of a simple $3$-polytope $P$, 
and $\mathcal{B}$ be a $k$-belt,  $k\geqslant 4$. Then in $M(P,\Lambda)$ the copies of the $k$-gon 
$P_k(\mathcal{B})\subset P$ form a disjoint union of $2^{r-r(\mathcal{B})}$  incompressible $2$-submanifolds
$M_x(\mathcal{B})$, $x\in \mathbb Z_2^r/\pi_\mathcal{B}$, with product neighbourhoods. Each submanifold 
$M_x(\mathcal{B})$ is homeomorphic to $M(P_k(\mathcal{B}),\Lambda_{\mathcal{B}})$, which is either $(T^2)^{\# g}$, $g=1+(k-4)2^{r(\mathcal{B})-3}$, or $(\mathbb RP^2)^{\#l}$, $l=2+(k-4)2^{r(\mathcal{B})-2}$.
\end{proposition}
\begin{remark}
For small covers another approach to the incompressibility of these submanifolds based on an explicit representation of the fundamental groups can be found in \cite{W19,WY17,LW20}.
\end{remark}
\begin{proof}[Proof of Proposition \ref{Binprop}]

Under the homeomorphism $\RZ_{K_P}\simeq\mathbb R\mathcal{Z}_P$ given in the proof of Proposition \ref{RZKP} 
the subspace $\mathbb{R}\mathcal{Z}_{K_{\omega(\mathcal{B})}}\times a\subset \RZ_{K_P}$ is mapped to the $2$-submanifold
$\mathbb{R}\mathcal{Z}_{P_k(\mathcal{B})}(a)$ in $\mathbb{R}\mathcal{Z}_P$ with a product neighbourhood. This submanifold
is homeomorphic to $\mathbb{R}\mathcal{Z}_{P_k(\mathcal{B})}$ and is glued 
of parts of polytopes $P\times b$ with $b$
lying in the coset in $\mathbb Z_2^m/\mathbb Z_2^{\mathcal{B}}$ corresponding to 
$a\in\mathbb Z_2^{m-|\omega(\mathcal{B})|}$, where 
$\mathbb Z_2^{\mathcal{B}}=\langle e_i\colon F_i\in\mathcal{B}\rangle\subset \mathbb Z_2^m$.

In $M(P,\Lambda)$ the copies of $P_k(\mathcal{B})$ are glued in a disjoint union of $2$-submanifolds 
$M_x(\mathcal{B})$ homeomorphic to 
$M(P_k(\mathcal{B}),\Lambda_{\mathcal{B}})$. By construction each $2$-submanifold $M_x(\mathcal{B})$ has a product neighbourhood and is glued of $k$-gons $P_k(\mathcal{B})\times b$, where $b$ belong to the coset 
$x+\pi_{\mathcal{B}}$ in $\mathbb Z_2^r/\pi_{\mathcal{B}}$.  
In particular, there are $2^{r-r(\mathcal{B})}$ such submanifolds. The submanifold $M_x(\mathcal{B})$ is 
an orbit space of a free action of the subgroup 
$$
H(\mathcal{B})=H(\Lambda)\cap \mathbb Z_2^{\mathcal{B}}={\rm Ker}\,\left[\Lambda_{\mathcal{B}}\colon \mathbb Z_2^{\mathcal{B}}\to\mathbb Z_2^m\right]
$$ 
on $\mathbb{R}\mathcal{Z}_{P_k(\mathcal{B})}(a)\subset\mathbb{R}\mathcal{Z}_P$ for any 
$a\in \widehat{\Lambda}^{-1}(x+\pi_{\mathcal{B}})$, where $\Lambda\colon\mathbb Z_2^m\to\mathbb Z_2^r$, 
and $\widehat{\Lambda}\colon \mathbb Z_2^m/\mathbb Z_2^{\mathcal{B}}\to \mathbb Z_2^r/\pi_{\mathcal{B}}$. 
In particular, there is a covering $\mathbb{R}\mathcal{Z}_{P_k(\mathcal{B})}(a)\to M_x(\mathcal{B})$ with the number of sheets 
$|H(\mathcal{B})|=2^{k-r(\mathcal{B})}$. 
 
According to Example \ref{ZPkex} each manifold $\mathbb R\mathcal{Z}_{P_k(\mathcal{B})}(a)$ is homeomorphic to 
a sphere with $g$ handles, $g=(k-4)2^{k-3}+1$. 

For an orientable manifold $M(P,\Lambda)$ each  submanifold 
$M_x(\mathcal{B})$ is also orientable, since it has a product neighbourhood. Thus, it is also a sphere with handles. Let $g'$ be its genus. From the covering  $\mathbb{R}\mathcal{Z}_{P_k(\mathcal{B})}(a)\to M_x(\mathcal{B})$ we see that $\chi(\mathbb R\mathcal{Z}_{P_k(\mathcal{B})})=|H(\mathcal{B})|\chi(M_x(\mathcal{B}))$. Therefore,  $2-2g=|H(\mathcal{B})|(2-2g')$, and 
$$
g'=1+\frac{g-1}{|H(\mathcal{B})|}=1+(k-4)\frac{2^{k-3}}{2^{k-r(\mathcal{B})}}=1+(k-4)2^{r(\mathcal{B})-3}.
$$

If $M(P,\Lambda)$ is non-orientable, then $M_x(\mathcal{B})$ may be either orientable or non-orientable. In the latter case 
it has the form 
$(\mathbb RP^2)^{\#l}$, where $2-2g=|H(\mathcal{B})|(2-l)$, and 
$$
l=2+2\frac{g-1}{|H(\mathcal{B})|}=2+2(k-4)\frac{2^{k-3}}{2^{k-r(\mathcal{B})}}=2+(k-4)2^{r(\mathcal{B})-2}.
$$

Now let us prove the incompressibility of the submanifolds $M_x(\mathcal{B})$.
Corollary \ref{Zwinc} implies that each submanifold 
$\mathbb{R}\mathcal{Z}_{P_k(\mathcal{B})}(a)\subset \mathbb R\mathcal{Z}_P$ is incompressible.

If $\Lambda$ has rank $m$, then $M_x(\mathcal{B})=\mathbb{R}\mathcal{Z}_{P_k(\mathcal{B})}(x)$.
Otherwise, consider the commutative diagrams of mappings and induced homomorphisms of groups:
$$
\xymatrix{
\mathbb R\mathcal{Z}_{P_k(\mathcal{B})}(a)\ar[r]^i\ar[d]^{q_1}&\mathbb{R}\mathcal{Z}_P\ar[d]^{q_2}\\
M_x(\mathcal{B})\ar[r]^j&M(P,\Lambda)
}\quad
\xymatrix{
\pi_1(\mathbb R\mathcal{Z}_{P_k(\mathcal{B})}(a))\ar[r]^{i_*}\ar[d]^{(q_1)_*}&\pi_1(\mathbb{R}\mathcal{Z}_P)\ar[d]^{(q_2)_*}\\
\pi_1(M_x(\mathcal{B}))\ar[r]^{j_*}&\pi_1(M(P,\Lambda))
}
$$

It is a classical fact that the projection $p\colon X\to X/G$ 
for a free action of a finite group $G$ on a Hausdorff topological space 
$X$ is a normal covering, $p_*\colon\pi_1(X)\to\pi_1(X/G)$ 
is injective, and $G\simeq \pi_1(X/G)/p_*\pi_1(X)$ (see \cite[Proposition 1.40]{H02}).
This concerns mappings $q_1$ and $q_2$.  In particular, $(q_1)_*$ and $(q_2)_*$ are injective, 
their images are normal subgroups, and 
$$
\pi_1(M_x(\mathcal{B}))/{\rm Im}\,(q_1)_*\simeq H(\mathcal{B})\simeq \mathbb Z_2^{k-r(\mathcal{B})},\quad \pi_1(M(P,\Lambda))/ {\rm Im}\,(q_2)_*\simeq H(\Lambda)\simeq \mathbb Z_2^{m-r}.
$$
Also $i_*$ in injective by Corollary \ref{Zwinc}.  
Since $j_*(q_1)_*$ is injective, ${\rm Ker}\,j_*\cap {\rm Im}\,(q_1)_*=\{1\}$. Then there is an injection
${\rm Ker}\,j_*\to \pi_1(M_x(\mathcal{B}))/{\rm Im}\,(q_1)_*\simeq H(\mathcal{B})$. If $k\geqslant 4$, then 
$\mathbb R\mathcal{Z}_{P_k(\mathcal{B})}(a)$ is a sphere with at least one handle. In particular, $\chi(\mathbb R\mathcal{Z}_{P_k(\mathcal{B})}(a))\leqslant 0$.
Since $q_1$ is a covering,  $\chi(M_x(\mathcal{B}))\leqslant 0$. It is a classical fact that a $2$-dimensional  compact
closed manifold $X$ with 
$\chi(X)\leqslant 0$ is a $K(\pi,1)$-space, that is $\pi_i(X)=0$ for $i\geqslant 2$. Also it is known that any element
in $\pi_1(K(\pi,1))$, where $K(\pi,1)$ is a finite cell complex, has an infinite order \cite[Proposition 2.45]{H02}. Since 
${\rm Ker}\,j_*$ consists of elements of finite order, it should be equal to $\{1\}$. This finishes the proof of Proposition \ref{Binprop}.
\end{proof}

\begin{example}\label{B4ex}
Consider a $4$-belt $\mathcal{B}=(F_i,F_j,F_k,F_l)$ of a simple $3$-polytope $P$. 
For the quadrangle $P_4(\mathcal{B})\simeq I\times I$
the real moment-angle manifold $\mathbb{R}\mathcal{Z}_{P_4}$ is a torus $T^2$ glued from $16$ copies of $P_4$. 
In $\mathbb{R}\mathcal{Z}_P$ it corresponds to $2^{m-4}$ disjoint incompressible tori. 

For  $M(P,\Lambda)$ defined by a vector-colouring of rank $r$ there are the following possibilities:
\begin{enumerate}
\item ${\rm rk}\,\Lambda_{\mathcal{B}}=4$. Then $M_x(\mathcal{B})\simeq \mathbb{R}\mathcal{Z}_{P_4}\simeq T^2$. 
\item ${\rm rk}\,\Lambda_{\mathcal{B}}=3$. Then either any three vectors form a basis or three subsequent vectors span a 
$2$-dimensional subspace, and the forth does dot lie in this subspace. 
In the second possibility either two opposite vectors coincide, or not.
Thus, up to a symmetry of the quadrangle we can assume that $\Lambda_i$, $\Lambda_j$ and $\Lambda_k$
is a basis, and $\Lambda_l\in \{\Lambda_i+\Lambda_j+\Lambda_k,\Lambda_j,\Lambda_j+\Lambda_k\}$.
The first two cases give $M_x(\mathcal{B})=T^2$, and the last case gives $M_x(\mathcal{B})=K^2$.
\item ${\rm rk}\,\Lambda_{\mathcal{B}}=2$. Then $\Lambda_i$ and $\Lambda_j$ form a basis, and  
either $\{\Lambda_k,\Lambda_l\}=\{\Lambda_i,\Lambda_j\}$, or $\Lambda_i+\Lambda_j\in\{\Lambda_k,\Lambda_l\}$. 
In the first case $\Lambda_k=\Lambda_i$, $\Lambda_l=\Lambda_j$, and $M_x(\mathcal{B})=T^2$, and in the second case 
up to a symmetry of the quadrangle we can assume that $\Lambda_k=\Lambda_i$, and $\Lambda_l=\Lambda_i+\Lambda_j$.    
Then $M_x(\mathcal{B})=K^2$.  
\end{enumerate}
 
\end{example}

We also need to consider $2$-submanifolds in $M(P,\Lambda)$ corresponding to facets of a simple $3$-polytope $P$
surrounded by belts.

\begin{definition}
Let $F_i$ be a $k$-gonal facet of a simple $3$-polytope $P$ surrounded by a  $k$-belt $\mathcal{B}$. Denote $\pi_{F_i}=\langle \Lambda_i\rangle+\pi_{\mathcal{B}}$, and $r(F_i)=\dim \pi_{F_i}$.
\end{definition}

\begin{proposition}\label{Finprop}
Let $M(P,\Lambda)$ be a manifold defined by a vector-colouring $\Lambda$ of rank $r$ of a simple $3$-polytope $P$, 
and $F_i$ be a facet surrounded by a $k$-belt $\mathcal{B}$ with $k\geqslant 4$. 
Then in $M(P,\Lambda)$ the copies of the $k$-gon $F_i$ form a disjoint union of $2^{r-r(F_i)}$ incompressible $2$-submanifolds
$M_x(F_i)$, $x\in \mathbb Z_2^r/\pi_{F_i}$. If $\Lambda_i\notin \pi_{\mathcal{B}}$, then
each manifold has a product neighbourhood with two boundary components $M_x(\mathcal{B})$ and $M_{x+\Lambda_i}(\mathcal{B})$ isotopic to $M_x(F_i)$. If $\Lambda_i\in \pi_{\mathcal{B}}$, then each manifold has a nontrivial 
tubular neighbourhood with the boundary $M_x(\mathcal{B})$. This neighbourhood is homeomorphic to a mapping
cylinder of a quotient map of a free action of an involution on $M_x(\mathcal{B})$ corresponding to the vector $\Lambda_i$.
\end{proposition}
\begin{remark}
In the case of small covers over $3$-polytope the incompressibility of facet submanifolds follows from \cite[Theorem 3.3]{WY17}. 
Another approach to this result is represented in \cite{DO01,DJS98,D08} and is based on the fact that $M_x(F_i)$
is a totally geodesic hypersurface in $M(P,\Lambda)$ with the structure of a cubical complex, which is nonpositively curved in the sense of Alexandrov and Gromov \cite{G87}. 
\end{remark}
\begin{proof}[Proof of Proposition \ref{Finprop}]

In $M(P,\Lambda)$ the copies of $F_i$ are glued in a disjoint union of 
$2^{r-r(F_i)}$ manifolds $M_x(F)$. Each manifold is glued of copies of $F_i$ corresponding to polytopes
$P\times b$ with $b$ lying in the coset $x+\pi_{F_i}\in \mathbb Z_2^r/\pi_{F_i}$. There is a piecewise-linear
homeomorphism between the part of the polytope $P$ between $P_k(\mathcal{B})$ and $F_i$ and $F_i\times I$. This 
homeomorphism gives an isotopy between  $P_k(\mathcal{B})$ and $F_i$. 
 
There are two possibilities:
either $\Lambda_i\not\in \pi_{\mathcal{B}}$, or $\Lambda_i\in \pi_{\mathcal{B}}$. In the first case 
$M_x(F_i)$ has a product tubular neighbourhood with the boundary consisting of
two manifolds $M_x(\mathcal{B})$ and $M_{x+\Lambda_i}(\mathcal{B})$. In particular, all the three manifolds
are isotopic and homeomorphic to $M(F_i,\Lambda_{\mathcal{B}})$, and $M_x(F_i)$ is incompressible
by Proposition \ref{Binprop}.

In the second case $M_x(F_i)$ 
has a nontrivial tubular neighbourhood with the boundary $M_x(\mathcal{B})$. $M_x(F_i)$ is homeomorphic
to a factor space of $M(F_i,\Lambda_{\mathcal{B}})$ by the action of the involution given by 
$\Lambda_i\in\pi_{\mathcal{B}}\simeq \mathbb Z_2^{r(\mathcal{B})}$.
The action if free, since $\Lambda_i\notin \langle\Lambda_p,\Lambda_q\rangle$ for any vertex $F_i\cap F_p\cap F_q$ of $F_i$.
Then $M_x(F_i)$ is homeomorphic to the manifold $M(F_i,\Lambda_{F_i})$, where $\Lambda_{F_i}$ is a vector-colouring 
given by the composition $\{F_j\colon F_j\in\mathcal{B}\}\to\pi_{\mathcal{B}}\to \pi_{\mathcal{B}}/\langle\Lambda_i\rangle$. 
The structure of $I$-bundle of the tubular neighbourhood of $M_x(F_i)$ gives a homeomorphism 
of this neighbourhood and the mapping cylinder of the quotient mapping  
$M(F_i,\Lambda_{\mathcal{B}})\to M(F_i,\Lambda_{F_i})$ described above.
It corresponds to the mapping from the boundary 
$M_x(\mathcal{B})$ to the zero section $M_x(F_i)$ of the tubular neighbourhood. 
This mapping is homotopic to the identical mapping 
and is a $2$-sheeted covering.

Thus, the inclusion 
$i\colon M_x(\mathcal{B})\to M(P,\Lambda)$ up to a homotopy can be decomposed as a composition 
of the $2$-sheeted covering $c\colon M_x(\mathcal{B})\to M_x(F_i)$ and the inclusion $j\colon M_x(F_i)\to M(P,\Lambda)$.
We have the composition of the homomorphisms 
$$
\pi_1(M_x(\mathcal{B}))\xrightarrow{c_*} \pi_1(M_x(F_i))\xrightarrow{j_*} \pi_1(M(P,\Lambda)).
$$
The first mapping is injective due to the covering space property, and the composition is injective by Proposition \ref{Binprop}.
Also $\pi_1(M_x(F_i))/{\rm Im}\,c_*\simeq \mathbb Z_2$.
Thus, as in the case of belts, we obtain that ${\rm Ker}\, j_*\cap {\rm Im}\,c_*=\{1\}$. In particular, there is an
embedding ${\rm Ker}\, j_*\to \pi_1(M_x(F_i))/{\rm Im}\,c_*\simeq \mathbb Z_2$. As discussed in the proof of 
Proposition \ref{Binprop}, this implies that ${\rm Ker}\, j_*$ is trivial, if $\chi(M_x(F_i))\leqslant 0$. The last condition
is satisfied, since $M_x(F_i)$ is covered by $M_x(\mathcal{B})$, which has a nonpositive Euler characteristic for
$k\geqslant 4$. 
\end{proof}

\begin{example}\label{F4ex}
If $F_i$ is a quadrangle, then taking into account Example \ref{B4ex} we have the following possibilities: 
\begin{enumerate}
\item $\Lambda_i\not\in \pi_{\mathcal{B}}$. Then $M_x(F_i)\simeq M_x(\mathcal{B})\simeq M_{x+\Lambda_i}(F_i)$ is either 
$T^2$, or $K_2$.  If $M(P,\Lambda)$ is orientable, then only the case $T^2$ is possible.
\item $\Lambda_i\not\in \pi_{\mathcal{B}}$. Then $M_x(F_i)$ has a tubular neighbourhood with the boundary $M_x(\mathcal{B})$,
and there is a $2$-sheeted covering $M_x(\mathcal{B})\to M_x(F_i)$. There are the following possibilities:
$(M_x(\mathcal{B}), M_x(F_i))\in \{(T^2,T^2), (T^2,K^2), (K^2,K^2)\}$. If $M(P,\Lambda)$ is orientable, then $M_x(\mathcal{B})\simeq T^2$, since it has a product neighbourhood. The vector $\Lambda_i$
is a linear combination of an odd number of vectors from the base of the set $\{\Lambda_j\colon F_j\in \mathcal{B}\}$.
Therefore, the corresponding involution on $M_x(\mathcal{B})$ changes the orientation, and $M_x(F_i)\simeq K^2$.
\end{enumerate}
\end{example}

\subsubsection{The $JSJ$-decomposition and an explicit  geometrization.}\label{JSJlast}
To pass to the $JSJ$-decomposition of the orientable manifold $M(P,\Lambda)$ we first 
realize geometrically the combinatorial decomposition of item (1) of Theorem \ref{M-th} using Corollary \ref{4bplane}. 
Each planar section $\widehat{P}_4(\mathcal{B})$ of $P$ corresponding to a belt $\mathcal{B}$ 
is isotopic to the piecewise linear polygon $P_4(\mathcal{B})$. Hence, the corresponding submanifolds  
$\widehat{M}_x(\mathcal{B})$ and $M_x(\mathcal{B})$ are isotopic in $M(P,\Lambda)$. 
Thus, in $M(P,\Lambda)$ we obtain a disjoint family of incompressible tori $\widehat{M}_x(\mathcal{B})$
corresponding to $4$-belts $\mathcal{B}$ from the nested family, incompressible tori $M_x(F_i)$ corresponding to 
"free" quadrangles $F_i$ of almost Pogorelov polytopes with $r(F_i)>r(\mathcal{B}(F_i))$, and incompressible tori 
$\widehat{M}_x(\mathcal{B})$, which are boundaries of the tubular neighbourhoods of the incompressible submanifolds $M_x(F_i)\simeq K^2$ corresponding to "free" quadrangles $F_i$ of almost Pogorelov polytopes with $r(F_i)=r(\mathcal{B}(F_i))$. For a $4$-belt around a quadrangle the section  $\widehat{P}_4(\mathcal{B})$ may be easily chosen as a slight shift of the facet. 

Now we delete from the flag polytope $P\ne I^3$ all the sections $\widehat{P}_4(\mathcal{B})$ corresponding to the belts
from the decomposition (1) and delete all the "free quadrangles" of almost Pogorelov polytopes from the decomposition.
Then each part $Q_i$ corresponding to an almost Pogorelov polytope $R_i$ without adjacent quadrangles is 
homeomorphic to this polytope with all the quadrangles deleted
and to a right-angled polytope of finite volume in $\mathbb H^3$ with ideal points deleted. The realization of a right-angled
polytope of finite volume in $\mathbb H^3$ is unique up to isometries of $\mathbb H^3$. Each part $Q_j$ 
corresponding to a  
$k$-prism $R_j$, $k\geqslant 5$, is homeomorphic to this prism with some disjoint set of quadrangles deleted, and also to
a direct product of a polygon with some set of vertices deleted and the segment. The polygon may be realized as a right-angled
polygon with deleted points lying on the absolute. The realization is not unique. In the orientable manifold $M(P,\Lambda)$ 
this partition of the polytope corresponds to a partition of the manifold into pieces defined by the deletion of the incompressible tori
corresponding to the belts from (1) and incompressible tori and Klein bottles corresponding to "free quadrangles" of almost Pogorelov polytopes. Each piece is glued from right-angled polytopes of finite volume in $\mathbb H^2\times \mathbb R$ and $\mathbb H^3$. 



The vector-colouring $\Lambda$ of rank $r$ induces a vector-colouring $\Lambda_{Q_i}$ of each
piece $Q_i$, which we consider as a right-angled polytope of finite volume in $\mathbb H^3$ or $\mathbb H^2\times \mathbb R$.  

\begin{definition}
Denote by $\pi_{Q_i}\subset \mathbb Z_2^r$ a linear span of vectors $\Lambda_j$ corresponding to facets of $Q_i$, and $r_i=\dim \pi_{Q_i}$.
\end{definition}

In the orientable manifold $M(P, \Lambda)$ each peace $M_x(Q_i)$ corresponding to $Q_i$ is homeomorphic to 
$N(Q_i,\Lambda_{Q_i})$ and is glued of the polytopes $Q_i\times b$ with $b$ lying in 
the coset $x+\pi_{Q_i}\in\mathbb Z_2^r/\pi_{Q_i}$. There are $2^{r-r_i}$ such pieces. 
The manifold $M_x(Q_i)$ contains the torus $\widehat{M}_y(\mathcal{B})$
from the above collection in its boundary 
if and only if $\widehat{P}_4(\mathcal{B})$ lies in the boundary of $Q_i$, and $y+\pi_{\mathcal{B}}\subset x+\pi_{Q_i}$.  
It contains the torus or the Klein bottle $M_y(F_j)$ from the above collection in its closure in 
$M(P,\Lambda)$ if and only 
$F_j$ is a free quadrangle of the almost Pogorelov polytope $R_i$, and $y+\pi_{\mathcal{B}}\subset x+\pi_{Q_i}$.
This gives an explicit geometrization of the pieces of the manifold $M(P,\Lambda)$.

Now our goal is to prove that this is actually the minimal geometric decomposition from Theorem \ref{GDT}. For this 
we will use Proposition \ref{GDP}. First using Proposition \ref{JSJcrit} we will prove that the collection 
of incompressible tori from Theorem \ref{M-th} gives the $JSJ$-decomposition, and then we will prove that $M(P,\Lambda)$ is not a $Sol$-manifold.

\begin{proposition}\label{MPprop}
For an orientable manifold $M(P,\Lambda)$ over a flag polytope $P$ different from $k$-prisms, 
the closure of each piece $M_x(Q_i)$
corresponding to a $k$-prism, $k\geqslant 5$, 
is a Seifert fibered manifold with toric boundary, and it is not homeomorphic to 
$D^2\times S^1$, $T^2\times I$ and $K^2\widetilde{\times} I$.
Moreover, for any two pieces $M_x(Q_i)$ and $M_y(Q_j)$ with a common boundary torus $T$ we have 
$f(T,\overline{M_x(Q_i)})\ne f(T,\overline{M_y(Q_j)})$.
\end{proposition}
\begin{remark}
If $P$ is a $k$-prism, $k\geqslant 4$, then the same argument as in the proof of Proposition \ref{MPprop}
gives a structure of a closed Seifert fibered manifold on an orientable manifold $M(P,\Lambda)$.
\end{remark}
\begin{proof}[Proof of Proposition \ref{MPprop}]
First consider the case $\Lambda=E$, and $M(P,E)=\mathbb R\mathcal{Z}_P$.

Consider the $k$-gon $L$, which is the base of the prism $S=L\times I$. 
There are $l$ of its edges corresponding to 
the canonical $4$-belts of $P$ (we will call these edges "singular").  Then copies of each base $L$ of $S$ in $\RZ_P$ are glued 
into several copies of an oriented $2$-dimensional submanifold $N$, which is homeomorphic to a sphere 
$S^2_{g,s}$ with $g$ handles and $s$ holes.  
Denote by $L'$ the polygon obtained by contraction of $l$ singular edges of $L$, and $k'=k-l$.  
It is easy to see that $k'\geqslant 3$.
Denote by $\widehat{N}\simeq S^2_g$ the
corresponding manifold without holes. Then $\widehat{N}\simeq\mathbb{R}\mathcal{Z}_{L'}$. In 
particular, $g=1+(k'-4)2^{k'-3}$ according to Example \ref{ZPkex}. Each singular 
edge of $L$ corresponds to a disjoint set of $2^{k'-2}$ holes with the boundary of each hole consisting of $4$ edges. 
Thus,  there are $s=l2^{k'-2}$ holes. Then in $\mathbb{R}\mathcal{Z}_P$ the $k$-prism corresponds to a disjoint union of 
submanifolds homeomorphic to $S^2_{g,s}\times S^1$. 
We have $l\geqslant 2$ for $k'=3$, and $l\geqslant 1$ for $k'\geqslant 4$. Therefore, $s\geqslant 4$.

\begin{lemma}\label{Sgslemma}
The manifold $S^2_{g,s}\times S^1$ with $g\geqslant 0$ and $s\geqslant 3$ can not cover a manifold 
homotopy equivalent to $D^2\times S^1$, $T^2\times I$ and $K^2\widetilde{\times} I$.
\end{lemma}
\begin{proof}
Set $X=S^2_{g,s}\times S^1$. The manifold $S^2_{g,s}$ is homotopy equivalent to a wedge of $p=2g+s-1\geqslant 2$ circles.
In particular, $\pi_1(X)=F_p\times \mathbb Z$, where $F_p$ is a free group on $p$ generators, and there is a
subgroup in $\pi_1(X)$ isomorphic to $F_2$. If  $\varphi\colon X\to Y$ is a covering, then $\varphi_*$ is an injection, and $\pi_1(Y)$
contains a subgroup $G$ isomorphic to $F_2$. Then $Y$ is not homotopy equivalent to $D^2\times S^1$ and $T^2\times I$, since 
their fundamental groups are abelian. If $Y$ is homotopy equivalent to $K^2\widetilde{\times} I$,
then $\pi_1(Y)\simeq \pi_1(K^2)$ and we will identify these groups. 
Consider the $2$-sheeted covering $\psi\colon T^2\to K^2$. 
This corresponds to the short exact sequence $1\to\pi_1(T^2)\to \pi_1(K^2)\to \mathbb Z_2\to 1$.
Since $G$ is not abelian, it is not contained in $\psi_*(\pi_1(T^2))$, and there is a short exact sequence 
$1\to{\rm Im}\,\psi_*\cap G\to G\to \mathbb Z_2\to 1$, 
which shows that the subgroup ${\rm Im}\,\psi_*\cap G$ has index $2$ in $G$.
Then it corresponds to a $2$-sheeted covering $\zeta\colon Z\to S^1\vee S^1$,
where $Z$ is connected.  In this case $Z$ is homotopy equivalent to a wedge of $3$ circles, and  
${\rm Im}\,\psi_*\cap G\simeq F_3$. But the group ${\rm Im}\,\psi_*\cap G$ is abelian. A contradiction. 
This finishes the proof.
\end{proof}

Thus, each piece of $\mathbb R\mathcal{Z}_P$ corresponding to a $k$-prism, is not homeomorphic to 
$D^2\times S^1$, $T^2\times I$ and $K^2\widetilde{\times} I$, and has a unique Seifert fibered structure by 
Theorem \ref{sun-th}. If two pieces $M_1$ and $M_2$ have a common boundary torus $T$, 
they correspond to different prisms.
Since the prisms are twisted, the elements $f(T,M_1)$ and $f(T,M_2)$ in $PH_1(T^2)$ 
correspond to the classes of  $(1,0)$ and $(0,1)$, and are different. 

Now consider the case of any vector-colouring $\Lambda$. 
Each manifold $\overline{M_x(Q_i)}$ corresponding to a $k$-prism is homeomorphic to a quotient space of 
$S^2_{g,s}\times S^1$ by a freely acting subgroup isomorphic to $H(\Lambda)\cap \mathbb Z_2^{Q_i}$, where 
$\mathbb Z_2^{Q_i}$ is a linear span of basis vectors $e_j\in\mathbb Z_2^m$ corresponding to facets of $Q_i$ (remind that $Q_i$
is obtained from the prism $L\times I$ by the deletion of the quadrangles corresponding to canonical belts). The action 
moves fibers to fibers and induces the following Seifert fibered structure on $\overline{M_x(Q_i)}$.
Fibers are again formed by the copies of $z\times I\subset L\times I$. 
Denote by $F_p$ and $F_q$ facets of $P$ containing bases of the prism.  
Then in $\overline{M_x(Q_i)}$ the copies of $z\times I\subset L\times I$ form circles $C(z)$.

\begin{enumerate}
\item If the point $z$ lies in the interior of $L$, or in the interior of a singular edge, then there is a neighbourhood 
$U(z)$ such that the copies of $U(z)\times I$ are glued to a direct product $U(z)\times S^1$, where $S^1$
is glued of $2^{\dim \langle\Lambda_p,\Lambda_q\rangle}$ copies of $I$. Hence, the fiber 
$C(z)$ is regular. 

\item If $z$ lies  in the interior of a regular edge of $L$, or it is a common vertex of a regular edge and a 
singular edge, then the regular edge corresponds to some facet $F_a$ of $P$
with $\Lambda_a\notin\{\Lambda_p,\Lambda_q\}$.
Since $M(P,\Lambda)$ is orientable, either $\Lambda_p=\Lambda_q$, or $\Lambda_p\ne\Lambda_q$ and 
$\Lambda_a\notin\langle \Lambda_p,\Lambda_q\rangle$. In both cases there is a neighbourhood 
$U(z)$ such that the copies of $U(z)\times I$ are glued to a direct product $U(z)\times S^1$, where $S^1$
is glued of $2^{\dim \langle\Lambda_p,\Lambda_q\rangle}$ copies of $I$. Hence, the fiber $C(z)$ is regular.

\item If $z$ is a common vertex of two regular edges of $L$, then these edges correspond to facets $F_a$ and $F_b$ of $P$,
with $\Lambda_p,\Lambda_q\notin \langle \Lambda_a,\Lambda_b\rangle$. If $\Lambda_p=\Lambda_q$,
then again $C(z)$ has a trivial tubular neighbourhood. If $\Lambda_p\ne\Lambda_q$, then 
$\Lambda_q=\Lambda_a+\Lambda_b+\Lambda_p$, since $M(P,\Lambda)$ is orientable. 
In this case $C(z)$ has a tubular neighbourhood 
that forms a standard fibered torus corresponding to the rotation by angle $\pi$.
\end{enumerate}

Lemma \ref{Sgslemma} implies that $\overline{M_x(Q_i)}$ is not not homeomorphic to 
$D^2\times S^1$, $T^2\times I$ and $K^2\widetilde{\times} I$, and has a unique Seifert fibered structure by 
Theorem \ref{sun-th}.  If two pieces $M_x(Q_i)$ and $M_y(Q_j)$ have a common boundary torus $T$, 
they correspond to different prisms. Since the prisms are twisted, Example \ref{B4ex}
shows that the elements $f(T,M_1)$ and $f(T,M_2)$ in $PH_1(T^2)$ are different. 
\end{proof}

Now we are ready to finish the proof of Theorem \ref{M-th}.

By Theorem \ref{hypnoSei} pieces corresponding to almost Pogorelov polytopes without adjacent quadrangles 
do not have a structure of a Seifert fibered manifold, and all the Seifert fibered pieces of $M(P,\Lambda)$ 
correspond to $k$-prisms and to Klein bottles arising from some free quadrangles of almost Pogorelov polytopes.
The latter pieces have a  common boundary only with hyperbolic pieces.
Thus, Propositions \ref{MPprop} and Proposition \ref{JSJcrit} imply that we indeed have a $JSJ$-decomposition.

To use Proposition \ref{GDP} we need to prove that orientable $M(P,\Lambda)$ can not be a $Sol$-manifold.
As it is mentioned in this Proposition, a $Sol$-manifold has one $JSJ$-torus, moreover, either both components of the complement of this torus are $K^2\widetilde{\times}I$, or the complement to the torus is homeomorphic to $T^2\times I$.
For $M(P,\Lambda)$ the hyperbolic pieces are not homeomorphic to $K^2\widetilde{\times}I$ and $T^2\times I$, since
they do not admit a Seifert fibered structure. Also we have proved that pieces corresponding to 
$k$-prisms are not homeomorphic to $K^2\widetilde{\times}I$ and $T^2\times I$. Only a piece corresponding
to a Klein bottle arising from a free quadrangle can 
be homeomorphic to $K^2\widetilde{\times}I$, but not to $T^2\times I$. Then the adjacent piece with 
the same boundary component is hyperbolic, which is a contradiction. This finishes the proof of Theorem \ref{M-th}.

\begin{corollary}
For orientable manifolds $M(P,\Lambda)$ defined by vector-colourings of simple $3$-polytopes the following five of eight Thurston's geometries arise:
\begin{itemize}
\item $S^3$ for  the simplex $\Delta^3$;
\item $S^2\times \mathbb R$ for the $3$-prism $\Delta^2\times I$;
\item  $\mathbb R^3$ for the cube $I^3$;
\item $\mathbb{H}^2\times\mathbb R$ for $k$-prisms, $k\geqslant 5$, and pieces corresponding to them;
\item $\mathbb H^3$ for Pogorelov polytopes and pieces corresponding to almost Pogorelov polytopes.
\end{itemize}
\end{corollary}

\section{Acknowledgements}
This work is supported by the Russian Science Foundation under grant 20-11-19998.
The  author is grateful to Victor Buchstaber for his encouraging support and attention to this work, to Taras Panov
for his advises, and to Vladimir Shastin and Dmitry Gugnin for fruitful discussions.

\end{document}